\documentclass[a4paper,12pt]{amsart}

\usepackage{soul}
\usepackage{amsmath, amssymb, amsfonts, amsthm, amssymb}
\usepackage{euscript, mathrsfs, stackrel}
\usepackage{enumerate}
\usepackage{stmaryrd}
\usepackage[utf8]{inputenc}   
\usepackage[T1]{fontenc}      
\usepackage{ulem} 

\usepackage{pgfplots}
\pgfplotsset{compat=1.14} 

\usepackage{color}
\usepackage{graphicx}
\usepackage{hyperref}
\usepackage{wrapfig}
\usepackage{cleveref}

\theoremstyle{plain}

\numberwithin{equation}{section}

\newtheorem{theorem}{Theorem}[section]
\newtheorem{definition}[theorem]{Definition}
\newtheorem{proposition}[theorem]{Proposition}
\newtheorem{lemma}[theorem]{Lemma}
\newtheorem{fact}[theorem]{Fact}

\newtheorem{corollary}[theorem]{Corollary}
\newtheorem{remark}[theorem]{Remark}

\usepackage[left=2.5cm,right=2.5cm,vmargin=2.5cm]{geometry}

\newcommand{\N}{\mathbb{N}}
\newcommand{\Z}{\mathbb{Z}}
\newcommand{\R}{\mathbb{R}}

\DeclareFontFamily{U}{mathx}{}
\DeclareFontShape{U}{mathx}{m}{n}{<-> mathx10}{}
\DeclareSymbolFont{mathx}{U}{mathx}{m}{n}
\DeclareMathAccent{\widehat}{0}{mathx}{"70}
\DeclareMathAccent{\widecheck}{0}{mathx}{"71}
\renewcommand{\check}[1]{\widecheck{#1}}

\numberwithin{equation}{section}

\DeclareMathOperator{\E}{\mathbb{E}}

\renewcommand{\P}{\mathbb{P}}

\newcommand{\dd}{\mathrm{d}}

\renewcommand{\bar}[1]{\overline{#1}}
\newcommand{\eqdist}{\stackrel{\rm (d)}{=}}
\renewcommand{\tilde}[1]{\widetilde{#1}}

\renewcommand{\rho}{\varrho}
\renewcommand{\epsilon}{\varepsilon}

\title{The Derrida-Retaux model on a geometric Galton--Watson tree}

\author{Gerold Alsmeyer}
\address{Gerold Alsmeyer, Institut für Mathematische Stochastik,
Universität Münster, D-48149 Münster,
Germany }
\email{gerolda@math.uni-muenster.de}

\author{Yueyun Hu}
\address{Yueyun Hu,
LAGA, Universit\'e Paris XIII,  93430 Villetaneuse,
France}
\email{yueyun@math.univ-paris13.fr}

\author{Bastien Mallein}
\address{Bastien Mallein,
Institut de Math\'ematiques de Toulouse, UMR 5219,
Universit\'e de Toulouse,
UPS, F-31062 Toulouse Cedex 9, France}
\email{bastien.mallein@math.univ-toulouse.fr}

\date{\today}

\begin{document}

\begin{abstract}
We consider a generalized Derrida-Retaux model on a Galton-Watson tree with a geometric offspring distribution. For a class of recursive systems, including the Derrida-Retaux model with either a geometric or exponential initial distribution, we characterize the critical curve using an involution-type equation and prove that the free energy satisfies the Derrida-Retaux conjecture.
\end{abstract}



\maketitle

\section{Introduction}

The Derrida-Retaux model, henceforth referred to as the DR model, is a max-type recursive equation in distribution introduced by the physicists Derrida and Retaux \cite{derrida-retaux} as a toy model for studying the depinning transition in the limit of strong disorder of the  Poland--Scheraga model (introduced in \cite{PS1,PS2}). After a simple change of variables, the model can be described as a family of recursively defined probability distributions $(\nu_{n}, n\geq 1)$ on $\R_+$. Specifically, for all $n \in \Z_+$, we consider two independent random variables $Y_{n}^{(1)}$ and $Y_{n}^{(2)}$ with law $\nu_{n}$. Then the random variable $Y_{n+1}$ is defined as
\begin{equation}\label{iteration:m=2}
Y_{n+1}\ =\ \left(Y_{n}^{(1)}+Y_{n}^{(2)}-1\right)_{+},
\end{equation}
where $x_{+} := \max\{ x, \, 0\}$ denotes the positive part of $x\in\R$, and we denote by $\nu_{n+1}$ its distribution. This model was previously studied by Collet, Eckmann, Glaser, and Martin \cite{collet-eckmann-glaser-martin} for probability distributions on $\Z_+$, where it served as a toy model for the study of spin glasses.

One of the most notable features of the DR model is that its free energy follows an infinite-order Berezinskii-Kosterlitz-Thouless (BKT) phase transition. This phenomenon, observed in disordered systems, is characterized by the absence of a stable distributional fixed point under renormalization. This connection has attracted significant interest, as it provides insight into the critical thresholds of hierarchical pinning models.

The free energy of the model $(\nu_{n})_{n\ge 0}$, starting from an initial distribution $\nu_{0}=\nu$, is defined as
\begin{equation}
F_Y(\nu)\,:=\, \lim_{n\to\infty} 2^{-n} \int x \nu_{n}(\dd x)\ =\ \lim_{n\to\infty} 2^{-n} \E(Y_{n})\ \in\ [0, \infty),
\end{equation}
where $Y_{n}$ is a random variable with distribution $\nu_{n}$. Note that the existence of the above limit follows directly from the fact that $\E(Y_{n+1})\le  2 \E(Y_{n})$, as implied by equation \eqref{iteration:m=2}. The model is said to be \emph{unpinned} if $F_Y(\nu) =0$, and \emph{pinned} if $F_Y(\nu)>0$, in reference to the Poland--Scheraga model.

We study the phase transition of this system from the pinned to the unpinned regime, considering initial conditions of the form $(1-p)\delta_0 + p\nu$ as $p$ increases from $0$ to $1$.
The critical point of the system is defined as
$$ p_c\,:=\, \inf\{p>0: F_Y((1-p) \delta_{0}+p \nu) >0\}. $$
Collet et al. \cite{collet-eckmann-glaser-martin}give the value of the critical parameter $p_c$ if $\nu$ is $\Z_+$-valued. In particular, they show that, for $\nu=\delta_2$, the critical point is
$$ p_c\,=\,\frac{1}{5}. $$
No such formula is known when $\nu$ is not integer-valued -- for example for $\nu = \frac{1}{2}(\delta_{1/2} + \delta_2)$. For further details and references on the recursive equation \eqref{iteration:m=2}, we refer to \cite{derrida-retaux} and \cite{bz_DRsurvey}.

While the definition of the model $\{Y_{n}\}$ is relatively simple, it exhibits intricate behavior at criticality, leading to a BKT phase transition and making its rigorous analysis challenging, with many fundamental questions remaining unresolved. It is widely believed that for a broad class of recursive models, including those described by \eqref{iteration:m=2}, the hierarchical renormalization model, and the pinning model (see \cite{derrida-hakim-vannimenus}, \cite{giacomin-lacoin-toninelli}, and \cite{berger-giacomin-lacoin}), the transition at the critical point is of infinite order. Specifically, for the model given by equation \eqref{iteration:m=2}, Derrida and Retaux \cite{derrida-retaux} conjectured the existence of a constant $C>0$ such that
\begin{equation}
 \label{conjectureDR}
  F_{Y}((1-p) \delta_{0}+p \nu)\ =\ \exp \left(-(C+o(1))  (p-p_c)^{-1/2}\right),\quad\text{as }p\downarrow p_c.
\end{equation}

Naturally, instead of the sum $Y_{n}^{(1)}+Y_{n}^{(2)}$ in \eqref{iteration:m=2}, we may generalize the model by considering the sum $Y^{(1)}_{n}+...+Y^{(m)}_{n}$ for any integer $m\ge 2$, where $Y^{(1)}_{n}, ..., Y^{(m)}_{n}$ are independent copies of $Y_{n}$. In this case
the recursion becomes
\begin{equation}\label{iteration:m}
Y_{n+1} \eqdist (Y^{(1)}_{n}+...+Y^{(m)}_{n} -1)_+ \quad\text{  for all $n\ge 0$},
\end{equation}
where the law of $Y_{0}$ is given by $(1-p) \delta_{0}+ p\nu$, with $\nu$ a probability distribution on $\N$, and $\eqdist$ denotes the equality in law. The corresponding critical value $p_c$ will naturally depend on the initial distribution $\nu$. A weaker form of the conjecture, stated as equation \eqref{conjectureDR}, was proved in \cite{bmvxyz_conjecture_DR} for the model  given by equation \eqref{iteration:m}. Specifically, assuming that $\int x^3 m^x \nu(\dd x)< \infty$ (an integrability condition that is also necessary), it was shown that
\begin{equation}
 F_{Y}((1-p) \delta_{0}+p \nu)\ =\ \exp\big(\!-(p-p_c)^{-(1/2)+o(1)}\big),\quad\text{as $p\downarrow p_c$.} \label{conjectureDR-w}
\end{equation}
Moreover, Chen \cite{Chen-freeenergy} recently established the infinite differentiability of the function $p\mapsto F_{Y}((1-p) \delta_{0}+p \nu)$, thereby proving that the phase transition is of infinite order.

The upper bound for the free energy obtained in \cite{bmvxyz_conjecture_DR} is more precise than the expression in \eqref{conjectureDR-w}, as it shows that the free energy is bounded by $A e^{-\delta (p-p_c)^{1/2}}$ for some constants $A,\delta>0$. However,  removing the $o(1)$ term in the lower bound of \cite{bmvxyz_conjecture_DR}, or determining the exact constant in \eqref{conjectureDR}, remains considerably more challenging. Nevertheless, an exactly solvable version of a continuous-time generalization of the DR model was described in \cite{HMP}. The integrability of the continuous model permits a more detailed analysis of the phase transition near criticality, thereby confirming the corresponding version of \eqref{conjectureDR}. See \cite{dds-continu} for further studies on the  scaling limit of the critical tree in the continuous model. We also refer to \cite{CDDS20} for the associated partial differential equations.

This work aims to investigate and analyze a class of exactly solvable discrete-time Derrida-Retaux models and to prove the Derrida-Retaux conjecture \eqref{conjectureDR} for them. Specifically, we consider a class where the parameter $m$ in \eqref{iteration:m} is replaced by a random variable with a geometric distribution, independent of the sequence $(Y^{(j)}_{n})_{j\ge 1}$. From a recursive tree-based perspective, this modification corresponds to replacing a regular $m$-ary tree with a Galton-Watson tree that has a geometric offspring distribution.

More precisely, we consider the \textit{generalized DR model} $(\nu_{n})_{n\geq 0}$, which is recursively defined as follows. Let $\mathtt{R}$ be a geometric random variable with parameter $p$, written~short\-hand as $\mathtt{R} \eqdist \mathcal{G}(p)$, and let $\mathtt{Z}$ be an independent positive random variable. For $n\in\N$, let $(X_{n}^{(k)}, k\ge 1)$ be i.i.d.~random variables with common law $\nu_{n}$, independent of $(\mathtt{R},\mathtt{Z})$. Then $\nu_{n+1}$ is defined as the law of $X_{n+1}$, where
\begin{equation}\label{eqn:derridaRetauxGeneralized}
X_{n+1}\ =\ \left(\sum_{k=1}^{\mathtt{R}} X_{n}^{(k)}\,-\, \mathtt{Z}\right)_+ .
\end{equation}

In particular, if the laws of $\nu_{0}$ and $\mathtt{Z}$ are supported on $\mathbb{Z}_+$, this process can be interpreted as a parking system on a Galton-Watson tree as follows. Start with a Galton-Watson tree of height $n$ with a geometric offspring law of parameter $p$. Assign to each leaf \  (thus node in generation $n$) a random number of cars according to the law $\nu_{0}$, and to each internal node a random number of parking spots according to the law of $\mathtt{Z}$. Cars then drive toward the root, parking at the first available spot they encounter. The number of cars reaching the root without finding a suitable parking spot follows the law $\nu_{n}$. The process is supercritical if this number grows exponentially as $n$ becomes large, and is critical or subcritical if the number converges in probability to 0. For detailed studies of parking models on trees, see Goldschmidt and Przykucki~\cite{goldschmidt-przykucki}, Aldous et al.~\cite{aldous-contat-curien-henard}, Chen and Contat~\cite{chenContat}, and Contat and Curien~\cite{contat-curien}.

We are interested in the free energy of the generalized DR model $(X_{n})$, defined as
\begin{equation}\label{eqn:generalizedFreeEnergy}
F_{X}(\nu_{0})\ =\ \lim_{n\to\infty} p^n \int x\, \nu_{n}(\dd x)\ =\ \lim_{n\to\infty}p^n\E(X_{n})\  \in\ [0, \infty].
\end{equation}
Similarly to the case \eqref{iteration:m=2}, this limit exists because, from \eqref{eqn:derridaRetauxGeneralized}, we have the inequality $\E(X_{n+1})\le  \E(\mathtt{R}) \E(X_{n})$, with $\E(\mathtt{R})=p^{-1}$, thus $(p^n \E(X_{n}))$ is nonincreasing and nonnegative.

We will consider the generalized DR model \eqref{eqn:derridaRetauxGeneralized} for the following two families of initial distributions $\nu_{0}$:

\begin{itemize}\itemsep2pt
\item {\it Linear fractional distributions,} which are mixtures of the Dirac measure at $0$ and a geometric distribution;
\item {\it Continuous linear fractional distributions,} which are mixtures of the Dirac measure at $0$ and an exponential distribution on $\R_+$.
\end{itemize}
The special case where ${\tt Z}=1$ was recently studied by Li and Zhang in \cite{lz-asymp} and \cite{lz-scaling}. Our main results for these two families of generalized DR models are as follows: We first provide a characterization of the critical curve that separates the regions where $F_X(\nu_{0})>0$ and $F_X(\nu_{0})=0$. We then establish the Derrida--Retaux conjecture by proving \eqref{conjectureDR} for $F_X$ as $\nu_{0}$ approaches the critical value.

The precise statements are provided in Theorems~\ref{cor:h-LF} and~\ref{cor:h-exp}.
To the best of our knowledge, this is the first time the precise asymptotics of the free energy at criticality for a discrete-time Derrida-Retaux model have been computed. Notably, this result is not limited to integer-valued systems or to iterations where $\texttt{Z}=1$ a.s.

The main property of these two families of initial distributions allowing them to be exactly solvable is the following: the transformation described in~\eqref{eqn:derridaRetauxGeneralized} maps a continuous linear fractional distribution to a continuous linear fractional distribution and -- provided that $\mathtt{Z}$ is integer-valued -- a linear fractional distribution to a linear fractional distribution. Therefore, the study of the generalized DR model reduces to studying the time evolution of the two parameters characterizing the discrete or continuous linear fractional distribution. This is similar to the results of \cite{HMP} in which the study of the continuous-time DR model is reduced to the study of a two-dimensional differential system.

We show in the next section that the evolution of the pair of parameters describing the law of the DR model started with a mixture of the Dirac measure at $0$ and a geometric or exponential distribution can be represented (up to an explicit change of variables) by the following iteration
\begin{equation}\label{eqn:recursiveEquationReducedB}
(u_{0},v_{0})\in\R_+ \times \R \quad\text{and}\quad
\begin{pmatrix} u_{n+1}\\ v_{n+1}\end{pmatrix}\ =\ \begin{pmatrix} u_{n}\Psi(v_{n+1})\\ u_{n} +v_{n}\end{pmatrix},
\end{equation}
where $\Psi$ is a nonnegative nondecreasing function satisfying $\Psi(0)=\Psi'(0)=1$. More specifically, we work under the following assumptions for $\Psi$:
\begin{equation}\label{ass:psi}
\Psi : \R\to (0, \infty) \text{ is a bounded, nondecreasing $\mathcal{C}^2$ function with }\Psi(0)=\Psi'(0)=1.\tag{A}
\end{equation}

As a consequence, the phase transition observed in the DR model can be understood by examining the asymptotic properties of the sequence $(u_{n},v_{n})_{n\ge 0}$, which is done in Sections~\ref{sec:basics} to \ref{sec:DRconj}.

We begin with a few observations on the sequence $(u_n,v_n)_{n\ge 0}$ defined recursively in \eqref{eqn:recursiveEquationReducedB}. Since $\Psi$ is positive, the sequence $(v_n)_{n\ge 0}$ is non-decreasing. Moreover, since $\Psi(v)<1$ for $v<0$ and $\Psi(v)>1$ for $v>0$, it follows that $(u_n)_{n\ge 0}$ is decreasing for $n<N_0$ and increasing for $n\ge N_0$, where $N_0:=\sup\{n\in\N:\ v_n\le 0\}$. In particular, the fixed points of the dynamics \eqref{eqn:recursiveEquationReducedB} are given by $\{(0,v): v\in\R\}$. These fixed points are stable for $v<0$ and unstable for $v>0$. The point $(0,0)$ plays a distinguished role: it is the unique critical fixed point and the largest point on the horizontal axis that can arise as a limit point of the evolution. More precisely, we have
\begin{equation}
  \label{eqn:limits}
  \lim_{n\to \infty} (u_{n},v_{n})\ \in\ \{(0,v):v\in (-\infty,0)\} \cup \{(0,0)\} \cup \{(\infty,\infty)\},
\end{equation}
as will be shown in Lemma~\ref{lem:limits}. This decomposition of the possible asymptotic behaviors of the equation enables us to divide the parameter space into three distinct domains:
\begin{equation}\label{eqn:decompositionOfSpace}
\begin{split}
&\mathcal{P} := \left\{(u_{0},v_{0}) : \lim_{n\to\infty} v_{n}=\infty \right\},\quad
\mathcal{C} := \left\{(u_{0},v_{0}) : \lim_{n\to\infty} v_{n}=0 \right\} \\ \text{ and } \quad
&\mathcal{U} := \left\{(u_{0},v_{0}) : \lim_{n\to\infty} v_{n} <0 \right\}.
\end{split}
\end{equation}
These domains are inspired by the depinning transition of polymers, with $\mathcal{P}$ corresponding to the pinned state (associated with a positive free energy), and $\mathcal{U}$ corresponding to the unpinned state (associated with a null free energy). By analogy with the associated DR models, we refer to:
\begin{itemize}\itemsep2pt
\item $\mathcal{P}$ as the supercritical domain, where $v_n \to \infty$,  which corresponds to the weak convergence of $\nu_{n}$ to $\delta_\infty$,
\item $\mathcal{U}$ as the subcritical domain, where $\lim_{n \to \infty} v_n < 0$ and $\lim_{n \to \infty} u_n = 0$, so that $\nu_{n}$ converges weakly to $ \delta_{0}$,
\item $\mathcal{C}$ as the critical domain, which forms the boundary between $\mathcal{P}$ and $\mathcal{U}$, where $(u_n,v_n) \to (0,0)$ and $\nu_n $ converges weakly to $\delta_0$, at a critical rate.
\end{itemize}

Our first main result concerning the recursion \eqref{eqn:recursiveEquationReducedB} is about the crucial properties of the critical curve $h$, which describes the boundary of the set $\mathcal{P}$. This function is defined as follows:
\begin{equation}\label{eqn:defineCriticalCurve}
h(v)\ :=\ \inf\{ u\in\R_+ : (u,v)\in\mathcal{P}\}\quad\text{for all }v\in\R.
\end{equation}
We derive a functional equation that $h$ satisfies, along with several of its regularity properties and its asymptotic behavior near the critical point $(0,0)$. We finally show that the set $\mathcal{C}$ is the graph of $h$.

\begin{theorem}
\label{thm:criticalCurve}
Under the assumptions \eqref{ass:psi}, the function $h$ is nonincreasing, Lipschitz~continuous and satisfies $0\le  h(x)\le  (-x)_+$ for all $x\in\R$. Moreover, $h$ is the unique~nonzero solution to the functional equation
\begin{gather}\label{eqn:functionalEquationCriticalCurve}
h(x+h(x))\ =\ \Psi(x+h(x)) h(x)\quad\text{for all }x\in\R_{-},
\intertext{and satisfies}
h(x)\ \sim\ \frac{x^2}{2},\quad\text{as }x \uparrow 0.\label{h-asymp}
\end{gather}
Finally, the domains $\mathcal{P},\mathcal{C}$ and $\mathcal{U}$ can be characterized as follows:
\begin{equation}\label{eqn:caracterisationOfPCU}
\begin{split}
\mathcal{P}\,&=\,\{(u,v) : u>h(v)\}, \quad \mathcal{C}\,=\,\{(u,v) : u=h(v)\},\\
\text{ and} \quad\mathcal{U} \,&=\,\{ (u,v) : u<h(v)\}.
\end{split}
\end{equation}
\end{theorem}

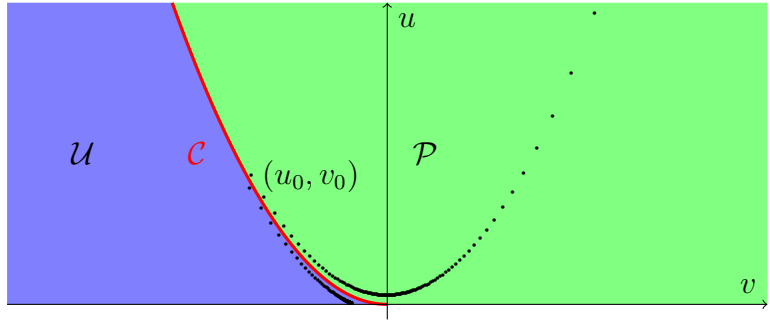
\begin{figure}[ht!]
\begin{tikzpicture}
\fill [color=blue!50] (0,0) -- (-5,0) -- (-5,4) -- (-2.828,4) -- plot[domain=-2.828:0] (\x,{\x*\x/2})  -- cycle;
\fill [color =green!50] (-2.828,4) -- plot[domain=-2.828:0] (\x,{\x*\x/2}) -- (5,0) -- (5,4)  -- cycle;
\draw [color=red, very thick, domain =-2.828:0, samples=200] plot (\x,{\x*\x/2});
\draw [color=red, very thick] (0,0)--(4,0);
\draw (-4,2) node {$\mathcal{U}$};
\draw (.5,2) node {$\mathcal{P}$};
\draw [color=red] (-2.5,2) node {$\mathcal{C}$};

\draw[->] (-5,0) -- (5,0) node[above left] {$v$};
\draw[->] (0,-.2) -- (0,4) node[below right] {$u$};
\draw (-1.79,1.703) node {$\cdot$}node[right] {$(u_{0},v_{0})$};
\draw (-1.62,1.414) node {$\cdot$};
\draw (-1.478,1.196) node {$\cdot$};
\draw (-1.359,1.027) node {$\cdot$};
\draw (-1.256,0.893) node {$\cdot$};
\draw (-1.167,0.786) node {$\cdot$};
\draw (-1.088,0.698) node {$\cdot$};
\draw (-1.018,0.624) node {$\cdot$};
\draw (-0.956,0.563) node {$\cdot$};
\draw (-0.9,0.511) node {$\cdot$};
\draw (-0.848,0.467) node {$\cdot$};
\draw (-0.802,0.429) node {$\cdot$};
\draw (-0.759,0.395) node {$\cdot$};
\draw (-0.719,0.366) node {$\cdot$};
\draw (-0.683,0.341) node {$\cdot$};
\draw (-0.649,0.318) node {$\cdot$};
\draw (-0.617,0.299) node {$\cdot$};
\draw (-0.587,0.281) node {$\cdot$};
\draw (-0.559,0.265) node {$\cdot$};
\draw (-0.532,0.251) node {$\cdot$};
\draw (-0.507,0.238) node {$\cdot$};
\draw (-0.484,0.226) node {$\cdot$};
\draw (-0.461,0.215) node {$\cdot$};
\draw (-0.439,0.206) node {$\cdot$};
\draw (-0.419,0.197) node {$\cdot$};
\draw (-0.399,0.189) node {$\cdot$};
\draw (-0.38,0.182) node {$\cdot$};
\draw (-0.362,0.175) node {$\cdot$};
\draw (-0.344,0.169) node {$\cdot$};
\draw (-0.328,0.164) node {$\cdot$};
\draw (-0.311,0.158) node {$\cdot$};
\draw (-0.295,0.154) node {$\cdot$};
\draw (-0.28,0.149) node {$\cdot$};
\draw (-0.265,0.145) node {$\cdot$};
\draw (-0.25,0.142) node {$\cdot$};
\draw (-0.236,0.138) node {$\cdot$};
\draw (-0.222,0.135) node {$\cdot$};
\draw (-0.209,0.132) node {$\cdot$};
\draw (-0.196,0.13) node {$\cdot$};
\draw (-0.183,0.127) node {$\cdot$};
\draw (-0.17,0.125) node {$\cdot$};
\draw (-0.157,0.123) node {$\cdot$};
\draw (-0.145,0.122) node {$\cdot$};
\draw (-0.133,0.12) node {$\cdot$};
\draw (-0.121,0.118) node {$\cdot$};
\draw (-0.109,0.117) node {$\cdot$};
\draw (-0.097,0.116) node {$\cdot$};
\draw (-0.086,0.115) node {$\cdot$};
\draw (-0.074,0.114) node {$\cdot$};
\draw (-0.063,0.113) node {$\cdot$};
\draw (-0.051,0.113) node {$\cdot$};
\draw (-0.04,0.112) node {$\cdot$};
\draw (-0.029,0.112) node {$\cdot$};
\draw (-0.018,0.112) node {$\cdot$};
\draw (-0.007,0.112) node {$\cdot$};
\draw (0.005,0.112) node {$\cdot$};
\draw (0.016,0.112) node {$\cdot$};
\draw (0.027,0.112) node {$\cdot$};
\draw (0.038,0.113) node {$\cdot$};
\draw (0.05,0.113) node {$\cdot$};
\draw (0.061,0.114) node {$\cdot$};
\draw (0.072,0.115) node {$\cdot$};
\draw (0.084,0.116) node {$\cdot$};
\draw (0.095,0.117) node {$\cdot$};
\draw (0.107,0.118) node {$\cdot$};
\draw (0.119,0.12) node {$\cdot$};
\draw (0.131,0.121) node {$\cdot$};
\draw (0.143,0.123) node {$\cdot$};
\draw (0.155,0.125) node {$\cdot$};
\draw (0.168,0.127) node {$\cdot$};
\draw (0.18,0.129) node {$\cdot$};
\draw (0.193,0.132) node {$\cdot$};
\draw (0.206,0.134) node {$\cdot$};
\draw (0.22,0.137) node {$\cdot$};
\draw (0.234,0.14) node {$\cdot$};
\draw (0.248,0.144) node {$\cdot$};
\draw (0.262,0.148) node {$\cdot$};
\draw (0.277,0.152) node {$\cdot$};
\draw (0.292,0.156) node {$\cdot$};
\draw (0.308,0.161) node {$\cdot$};
\draw (0.324,0.166) node {$\cdot$};
\draw (0.34,0.172) node {$\cdot$};
\draw (0.357,0.178) node {$\cdot$};
\draw (0.375,0.184) node {$\cdot$};
\draw (0.394,0.191) node {$\cdot$};
\draw (0.413,0.199) node {$\cdot$};
\draw (0.433,0.208) node {$\cdot$};
\draw (0.453,0.217) node {$\cdot$};
\draw (0.475,0.227) node {$\cdot$};
\draw (0.498,0.238) node {$\cdot$};
\draw (0.522,0.251) node {$\cdot$};
\draw (0.547,0.264) node {$\cdot$};
\draw (0.573,0.279) node {$\cdot$};
\draw (0.601,0.296) node {$\cdot$};
\draw (0.631,0.314) node {$\cdot$};
\draw (0.662,0.335) node {$\cdot$};
\draw (0.696,0.358) node {$\cdot$};
\draw (0.731,0.383) node {$\cdot$};
\draw (0.77,0.412) node {$\cdot$};
\draw (0.811,0.445) node {$\cdot$};
\draw (0.855,0.482) node {$\cdot$};
\draw (0.904,0.525) node {$\cdot$};
\draw (0.956,0.574) node {$\cdot$};
\draw (1.014,0.631) node {$\cdot$};
\draw (1.077,0.697) node {$\cdot$};
\draw (1.146,0.775) node {$\cdot$};
\draw (1.224,0.867) node {$\cdot$};
\draw (1.311,0.978) node {$\cdot$};
\draw (1.408,1.111) node {$\cdot$};
\draw (1.52,1.275) node {$\cdot$};
\draw (1.647,1.477) node {$\cdot$};
\draw (1.795,1.732) node {$\cdot$};
\draw (1.968,2.059) node {$\cdot$};
\draw (2.174,2.486) node {$\cdot$};
\draw (2.422,3.059) node {$\cdot$};
\draw (2.728,3.848) node {$\cdot$};

\draw (-1.81,1.537) node {$\cdot$};
\draw (-1.656,1.27) node {$\cdot$};
\draw (-1.529,1.067) node {$\cdot$};
\draw (-1.423,0.908) node {$\cdot$};
\draw (-1.332,0.783) node {$\cdot$};
\draw (-1.254,0.681) node {$\cdot$};
\draw (-1.185,0.598) node {$\cdot$};
\draw (-1.126,0.528) node {$\cdot$};
\draw (-1.073,0.47) node {$\cdot$};
\draw (-1.026,0.42) node {$\cdot$};
\draw (-0.984,0.378) node {$\cdot$};
\draw (-0.946,0.341) node {$\cdot$};
\draw (-0.912,0.309) node {$\cdot$};
\draw (-0.881,0.281) node {$\cdot$};
\draw (-0.853,0.257) node {$\cdot$};
\draw (-0.827,0.235) node {$\cdot$};
\draw (-0.804,0.216) node {$\cdot$};
\draw (-0.782,0.199) node {$\cdot$};
\draw (-0.762,0.183) node {$\cdot$};
\draw (-0.744,0.169) node {$\cdot$};
\draw (-0.727,0.157) node {$\cdot$};
\draw (-0.711,0.145) node {$\cdot$};
\draw (-0.697,0.135) node {$\cdot$};
\draw (-0.683,0.126) node {$\cdot$};
\draw (-0.671,0.117) node {$\cdot$};
\draw (-0.659,0.109) node {$\cdot$};
\draw (-0.648,0.102) node {$\cdot$};
\draw (-0.638,0.095) node {$\cdot$};
\draw (-0.628,0.089) node {$\cdot$};
\draw (-0.62,0.084) node {$\cdot$};
\draw (-0.611,0.078) node {$\cdot$};
\draw (-0.603,0.074) node {$\cdot$};
\draw (-0.596,0.069) node {$\cdot$};
\draw (-0.589,0.065) node {$\cdot$};
\draw (-0.583,0.061) node {$\cdot$};
\draw (-0.576,0.058) node {$\cdot$};
\draw (-0.571,0.054) node {$\cdot$};
\draw (-0.565,0.051) node {$\cdot$};
\draw (-0.56,0.048) node {$\cdot$};
\draw (-0.555,0.046) node {$\cdot$};
\draw (-0.551,0.043) node {$\cdot$};
\draw (-0.546,0.041) node {$\cdot$};
\draw (-0.542,0.038) node {$\cdot$};
\draw (-0.539,0.036) node {$\cdot$};
\draw (-0.535,0.034) node {$\cdot$};
\draw (-0.532,0.032) node {$\cdot$};
\draw (-0.528,0.031) node {$\cdot$};
\draw (-0.525,0.029) node {$\cdot$};
\draw (-0.522,0.028) node {$\cdot$};
\draw (-0.52,0.026) node {$\cdot$};
\draw (-0.517,0.025) node {$\cdot$};
\draw (-0.514,0.023) node {$\cdot$};
\draw (-0.512,0.022) node {$\cdot$};
\draw (-0.51,0.021) node {$\cdot$};
\draw (-0.508,0.02) node {$\cdot$};
\draw (-0.506,0.019) node {$\cdot$};
\draw (-0.504,0.018) node {$\cdot$};
\draw (-0.502,0.017) node {$\cdot$};
\draw (-0.5,0.016) node {$\cdot$};
\draw (-0.499,0.015) node {$\cdot$};
\draw (-0.497,0.015) node {$\cdot$};
\draw (-0.496,0.014) node {$\cdot$};
\draw (-0.494,0.013) node {$\cdot$};
\draw (-0.493,0.013) node {$\cdot$};
\draw (-0.492,0.012) node {$\cdot$};
\draw (-0.491,0.011) node {$\cdot$};
\draw (-0.49,0.011) node {$\cdot$};
\draw (-0.488,0.01) node {$\cdot$};
\draw (-0.487,0.01) node {$\cdot$};
\draw (-0.486,0.009) node {$\cdot$};
\draw (-0.486,0.009) node {$\cdot$};
\draw (-0.485,0.008) node {$\cdot$};
\draw (-0.484,0.008) node {$\cdot$};
\draw (-0.483,0.008) node {$\cdot$};
\draw (-0.482,0.007) node {$\cdot$};
\draw (-0.482,0.007) node {$\cdot$};
\draw (-0.481,0.007) node {$\cdot$};
\draw (-0.48,0.006) node {$\cdot$};
\draw (-0.48,0.006) node {$\cdot$};
\draw (-0.479,0.006) node {$\cdot$};
\draw (-0.478,0.005) node {$\cdot$};
\draw (-0.478,0.005) node {$\cdot$};
\draw (-0.477,0.005) node {$\cdot$};
\draw (-0.477,0.005) node {$\cdot$};
\draw (-0.476,0.004) node {$\cdot$};
\draw (-0.476,0.004) node {$\cdot$};
\draw (-0.476,0.004) node {$\cdot$};
\draw (-0.475,0.004) node {$\cdot$};
\draw (-0.475,0.004) node {$\cdot$};
\draw (-0.474,0.003) node {$\cdot$};
\draw (-0.474,0.003) node {$\cdot$};
\draw (-0.474,0.003) node {$\cdot$};
\draw (-0.474,0.003) node {$\cdot$};
\draw (-0.473,0.003) node {$\cdot$};
\draw (-0.473,0.003) node {$\cdot$};
\draw (-0.473,0.003) node {$\cdot$};
\draw (-0.472,0.002) node {$\cdot$};
\draw (-0.472,0.002) node {$\cdot$};
\draw (-0.472,0.002) node {$\cdot$};
\draw (-0.472,0.002) node {$\cdot$};
\draw (-0.472,0.002) node {$\cdot$};
\draw (-0.471,0.002) node {$\cdot$};
\draw (-0.471,0.002) node {$\cdot$};
\draw (-0.471,0.002) node {$\cdot$};
\draw (-0.471,0.002) node {$\cdot$};
\draw (-0.471,0.002) node {$\cdot$};
\draw (-0.47,0.001) node {$\cdot$};
\draw (-0.47,0.001) node {$\cdot$};
\draw (-0.47,0.001) node {$\cdot$};
\draw (-0.47,0.001) node {$\cdot$};
\end{tikzpicture}
\caption{\small Decomposition of the phase space for a function $\Psi$ defined by $\Psi(x)= x^2\big/{2(1+x-\sqrt{1+2x})}$ for $x \in [-0.5,0.5]$, extended to $\R$ in such a way that \eqref{ass:psi} holds. For this function, the critical curve $h$ is given by $x \mapsto x^2/2$ on $[-0.5,0]$ and drawn in red.  Additionally, slightly supercritical and subcritical trajectories for $(u,v)$ are depicted.} \label{fig:phaseDecomposition}
\end{figure}

We do not have a way to construct explicit solutions to the involution-type equation \eqref{eqn:functionalEquationCriticalCurve}, even for some natural choices of the function $\Psi$ associated with the stochastic equations described in Section~\ref{sec:stochastique}. However, it is relatively straightforward to first select $h$ and then choose a function $\Psi$ such that $h$ satisfies \eqref{eqn:functionalEquationCriticalCurve}, as demonstrated in Figure~\ref{fig:phaseDecomposition}.

We use the function $h$ to quantify the distance from a given starting point to the critical curve and introduce an analogue of the free energy for the recursion \eqref{eqn:recursiveEquationReducedB} that allows us to state and prove an analog of the Derrida-Retaux conjecture. Specifically, we define
$$ \Psi(\infty)\,=\,\lim_{x\to \infty} \Psi(x), $$
which always exists and is finite as $\Psi$ is assumed to be monotone and bounded. The analog of the free energy associated with the recursion \eqref{eqn:recursiveEquationReducedB} is defined by
\begin{equation}\label{eqn:freeEnergyb}
F(u_{0},v_{0})\ :=\ \liminf_{n \to\infty} \Psi(\infty)^{-n} u_{n}\ \in\ [0,\infty).
\end{equation}
We will show that this quantity corresponds to the free energy of the generalized Derrida--Retaux model. Some general observations on this notion of free energy are presented in the next result.

\begin{proposition}
\label{prop:freeEnergy}
Under assumptions \eqref{ass:psi}, we have
$$ F(u_{0},v_{0})\ =\ \lim_{n\to\infty} \Psi(\infty)^{-n} u_{n}\ =\ \inf_{n\in\N} \Psi(\infty)^{-n} u_{n}. $$
Furthermore, if
\begin{equation}\label{ass:technical}
\sum_{j=1}^\infty (\Psi(\infty)-\Psi(\kappa^j))<\infty\quad\text{for all }\kappa>1,\tag{B}
\end{equation}
then we have the equivalence
$$ (u_{0},v_{0})\in\mathcal{P}\ \iff\ F(u_{0},v_{0})>0. $$
\end{proposition}

Observe that the (primarily technical) assumption \eqref{ass:technical} expresses that $\Psi(x)$ converges sufficiently fast to $\Psi(\infty)$ as $x\to \infty$. This is implied, for example, by the condition
$$ \lim_{x\to \infty} (\log x)^{1+\delta} \left(\Psi(x)-\Psi(\infty)\right)=0\quad\text{for some }\delta>0. $$
Under assumption \eqref{ass:technical}, we can identify the supercritical domain with the set of parameters such that the free energy is positive.

We can now state the Derrida-Retaux conjecture for the recursive equation \eqref{eqn:recursiveEquationReducedB} as follows.

\begin{theorem}
\label{thm:main}
Under the assumptions \eqref{ass:psi} and \eqref{ass:technical}, for all $v\le  0$, there exists ${\mathbb C}_{v}>0$ such that
$$ \lim_{\epsilon\to 0} \epsilon^{1/2} \log F(h(v)+\epsilon,v)\ =\ -{\mathbb C}_{v}. $$
Moreover, we have
\begin{equation}
{\mathbb C}_{0}\ =\ \frac{1}{2} \lim_{v\to 0-} {\mathbb C}_{v}\ =\ \frac{\pi}{\sqrt{2}} \log \Psi(\infty).
\end{equation}
\end{theorem}

The last statement particularly highlights the universal constant $\lim_{v\to 0-} \frac{{\mathbb C}_{v}}{\log \Psi(\infty)}=\pi \sqrt{2}$. It is also worth noting the partial symmetry breaking around the critical point, the function $v \in (-\infty, 0] \mapsto {\mathbb C}_v$ being discontinuous at $0$. With similar methods as the ones we develop here, one can show that for $v > 0$,  $F(\epsilon,v) = \epsilon^{\gamma_v + o(1)}$ as $\epsilon \to 0$, where $\gamma_v:= \frac{\log \Psi(\infty)}{\log \Psi(v)}\ge 1$.
These results are in agreement with the findings in \cite{HMP} for the differential system associated to the solvable continuous-time DR model.

Finally, we present a result describing the behavior of the recursion along the critical curve, in agreement with previous results on DR models (see \cite{bmvxyz_conjecture_DR,xyz_sustainability,xz_stable}).

\begin{theorem}
\label{thm:criticalEvolution}
Assuming \eqref{ass:psi}, let $v<0$ and $(u_{0},v_{0})=(h(v),v)$. Then
$$ u_{n}\,\sim\,\frac{2}{n^2} \quad \text{and}\quad v_{n}\,\sim\, -\frac{2}{n}\quad\text{as }n\to\infty. $$
\end{theorem}

We conclude this introduction with some comments. The generalized DR model \eqref{eqn:derridaRetauxGeneralized} in the case ${\tt Z}=1$ with a general initial distribution $\nu$ was studied in \cite{yz_bnyz}. In that work, the asymptotic behavior of the free energy was examined in a domain corresponding, in our normalized coordinates, to $u_{0}\to 0$ and $v_{0}>0$. With specific assumptions on the tail of $\nu$, different asymptotic behaviors from \eqref{conjectureDR} can emerge at $v_{0}=0$.

The  model \eqref{eqn:derridaRetauxGeneralized} with ${\tt Z}=1$, where the initial distribution is a mixture of a Dirac measure at 0 and either a geometric or exponential distribution, has been studied in Li and Zhang \cite{lz-asymp, lz-scaling}. These works classify different regimes and provide precise estimates for the distribution of $X_{n}$ in the critical regime; in this special case (${\tt Z}=1$), Theorem 1.4 is consistent with \cite[Theorem 1.3]{lz-asymp}, where the authors further obtain a higher-order expansion. In addition, \cite{lz-scaling} explores the scaling limit, which leads to a continuous-time model.

Previous studies of the stochastic recursions \eqref{iteration:m=2} or \eqref{iteration:m} have heavily relied on the explicit computation of the critical point $p_c$ as derived in Collet et al.~\cite{collet-eckmann-glaser-martin}. This computation requires two key assumptions: that $\nu_{0}$ is supported on $\Z_+$, and that the subtraction term ${\tt Z}$ equals $1$ almost surely. For instance, if the constant 1 in \eqref{iteration:m=2} or \eqref{iteration:m} is replaced by 2, even the computation of the critical value becomes an open problem; See \cite{KotLot24} for some estimates on the critical value.

Our method, however, proves to be robust. By reducing the study of the law of $X_{n}$ to that of the recursive equation \eqref{eqn:recursiveEquationReducedB}, we can establish universality without needing to know the explicit value of the critical point. The critical point is characterized by the function $h$, which is the solution of the involution-type equation \eqref{eqn:functionalEquationCriticalCurve}. Specifically, as long as the critical curve $h$ satisfies $h(x) \sim x^2/2$ as $x\to 0$, the free energy defined for the recursive equation \eqref{eqn:recursiveEquationReducedB} will satisfy the corresponding Derrida-Retaux conjecture.

Finally, we mention studies of the Derrida-Retaux system in other contexts: as a spin glass model in in Collet et al.~\cite{collet-eckmann-glaser-martin2}, as an iteration function of random variables in Li and Rogers~\cite{li-rogers} and Jordan \cite{jordan}, and as part of the max-type recursion families in the seminal paper by Aldous and Bandyopadhyay~\cite{aldous-bandyopadhyay}.

The rest of the article is organized as follows. In the next section, we introduce two families of two-parameter DR models, whose evolution is governed by equation \eqref{eqn:recursiveEquationReducedB}. Specifically, we discuss the implications of Theorems~\ref{thm:main} and~\ref{thm:criticalEvolution} for these processes. Section~\ref{sec:basics} explores some basic properties of the recursion \eqref{eqn:recursiveEquationReducedB}, including a proof of Proposition~\ref{prop:freeEnergy} and an analysis of the backward evolution of this equation. In Section~\ref{sec:criticalCurve}, we investigate the critical domain, proving Theorems~\ref{thm:criticalCurve} and~\ref{thm:criticalEvolution}. Finally, we establish Theorem~\ref{thm:main} in Section~\ref{sec:DRconj}, thereby confirming the Derrida-Retaux conjecture for \eqref{eqn:recursiveEquationReducedB}.

\paragraph{Notation}
Throughout, we use the following standard notation: $\N$ for the set of~positive integers, $\Z_+$ for the set of nonnegative integers, $\R_+$ for the set of nonnegative reals and $(0,\infty)$ the set of positive reals. We also recall that $x_+=\max(x,0)$ is the positive part of $x$, and we write $x_-=(-x)_+$ for the negative part of $x$.

\section{Solvable discrete-time DR models with two parameters}
\label{sec:stochastique}

In this section, we introduce two families of two-parameters DR models. As mentioned in the introduction, these models can be thought of as stochastic recursions on a Galton-Watson tree with a geometric offspring distribution. Using several equalities in distribution involving sums of a geometric number of i.i.d. random variables, we are able to describe the families of laws by tracking two parameters that evolve according to \eqref{eqn:recursiveEquationReducedB}.

Recall \eqref{eqn:derridaRetauxGeneralized}: For $n\in\N$, writing $(X_{n}^{(k)}, k\ge 1)$ for i.i.d. random variables of law $\nu_{n}$, independent of $(\mathtt{R},\mathtt{Z})$, then
$$
  X_{n+1}=\left( \sum_{j=1}^{\mathtt{R}} X_{n}^{(k)} -\mathtt{Z}\right)_+ \text{ is a random variable of law $\nu_{n+1}$.}
$$

We introduce in Section~\ref{subsec:drDiscret} a generalized DR model such that $\nu_{n}$ is a probability distribution on $\Z_+$ for all $n\in\N$, then in Section~\ref{subsec:drContinu} a generalized DR model such that $\nu_{n}{}_{|(0,\infty)}$ has density with respect to the Lebesgue measure on $(0,\infty)$. We relate these two models to the recursion equation \eqref{eqn:recursiveEquationReducedB}, which allows us to apply Theorems~\ref{thm:main} and~\ref{thm:criticalEvolution}.

\subsection{Generalized DR model on {$\Z_+$} with linear fractional input}
\label{subsec:drDiscret}

We assume in this section that $\mathtt{Z}$ takes values in $\N$, the set of positive integers. In this case, if $\nu_{0}$ is supported   on $\Z_+$, then $\nu_{n}$ will be supported   on $\Z_+$ for all $n\ge 1$. A solvable DR model will be obtained by choosing for the input distribution $\nu_{0}$ a {\it linear fractional distribution} that we now introduce, see also \cite{Alsmeyer:21} for further details about this class of distributions that is of particular interest in the theory of branching processes.

\begin{definition}[Linear fractional distribution]
\label{def-LF}
Let $\alpha, \beta>0$ be such that $\alpha+\beta \ge 1$ and $Y$ be a random variable. We say that $Y$ has a linear fractional distribution with para\-meters $\alpha, \beta$ and write $Y \eqdist \mathsf{LF}(\alpha, \beta)$ if $Y$ takes values in $\Z_+$ and has probability generating function $f_{Y}(s):=\E(s^{Y}) $ which satisfies
$$ \frac{1}{1-f_{Y}(s)}=\frac{\alpha}{1-s}+\beta , \quad \text{i.e.}\quad f_{Y}(s)=1- \frac{1-s}{\alpha+\beta(1-s)} $$
for $0\le s\le 1$. In particular, $\E(Y)= \alpha^{-1}$.
\end{definition}

In this section, we consider the generalized DR model \eqref{eqn:derridaRetauxGeneralized} with geometrically distributed $\tt R$ and initial distribution $\nu_{0}=\mathsf{LF}(\alpha,\beta)$, where $\alpha$ and $\beta$ are as specified. We demonstrate that, for each $n\in\N$, the distribution $\nu_{n}$ remains linear fractional, i.e.~$\nu_{n}=\mathsf{LF}(\alpha_{n},\beta_{n})$ for appropriate values of  $(\alpha_{n},\beta_{n})$ that will be further specified below. Additionally, we show that, up to a reparametrization, the evolution of this sequence can be mapped to the recursion \eqref{eqn:recursiveEquationReducedB} for a suitable function $\Psi$.

If $Y \eqdist \mathsf{LF}(\alpha, \beta)$, then we infer from Definition~\ref{def-LF} that
$$ \P(Y=k)=\frac{\alpha}{(\alpha+\beta)^2} \left( \frac{\beta}{\alpha+\beta} \right)^{k-1}\text{ for $k\ge 1$ and}\quad \P(Y=0)=1-\frac{1}{\alpha+\beta}. $$
The geometric distribution appears as a special case of linear fractional distribution, namely $\mathcal{G}(p)=\mathsf{LF}(p,1-p)$ for all $p \in (0,1)$.  Two well-known distributional identities for linear fractional distributions are next.

\begin{fact}
\label{f:1}
Let $\alpha, \beta>0$ be such that $\alpha+\beta\ge 1$ and ${\tt R} \eqdist \mathcal{G}(p)$ for some $p\in (0,1)$. If $Y_1,Y_{2},\ldots$ are i.i.d.~random variables with common law $\mathsf{LF}(\alpha, \beta)$ and independent of ${\tt R}$, then
$$ \sum_{j=1}^{\tt R} Y_{j}\ \eqdist\ \mathsf{LF}(p \alpha, 1-p+p \beta). $$
\end{fact}

\begin{proof} Denote by $f_{\tt R}, f_{Y}$ and $f_\Sigma$ the probability generating functions of ${\tt R}, Y_1$ and $\sum_{j=1}^{\tt R} Y_{j}$ respectively. Then $f_{\Sigma}(s)= f_{\tt R}(f_{Y}(s))$ for $0\le s \le 1$ and therefore
\begin{align*}
\frac1{1-f_\Sigma(s)}\ =\ \frac1{1- f_{\tt R}(f_{Y}(s))}\ &=\ 1-p+\frac{p}{1- f_{Y}(s)}\\
&=\  1-p+p \left( \beta+\frac{\alpha}{1-s}\right)\ =\ (1-p+p \beta)+\frac{p \alpha}{1-s},
\end{align*}
which completes the proof.
\end{proof}

\begin{fact}\label{f:2}
If $Y,{\tt Z}$ are independent integer-valued random variables such that $Y\eqdist \mathsf{LF}(\alpha, \beta)$, then
$$ (Y-{\tt Z} )_+\ \eqdist\ \mathsf{LF}\left(\frac{\alpha}{\varphi(\beta/(\alpha+\beta))},\frac{\beta}{\varphi(\beta/(\alpha+\beta))}\right), $$
where $\varphi(s)=\E(s^{\tt Z} )$ for $0\le s \le 1$.
\end{fact}

\begin{proof}
Since $\P(Y\geq k)=p_{0} \lambda^{k-1}$ for all $k\ge 1$, where $\lambda=\frac{\beta}{\alpha+\beta}$ and $p_{0}=\frac{1}{\alpha+\beta}$, it follows that
$$ \P((Y-{\tt Z} )_+\ge k)\ =\ \E\left( p_{0} \lambda^{k-1+{\tt Z} } \right)\ =\ p_{0} \lambda^{k-1} \varphi(\lambda). $$
and thus $(Y-{\tt Z} )_+\eqdist\mathsf{LF}(\gamma,\delta)$ with
$$ \frac{\delta}{\gamma+\delta}\ =\ \lambda\ =\ \frac{\beta}{\alpha+\beta}\quad\text{and}\quad \frac{1}{\gamma+\delta}\ =\ p_{0} \varphi(\lambda)\ =\ \frac{1}{\alpha+\beta}\,\varphi\left(\frac{\beta}{\alpha+\beta}\right). $$
The asserted values of $\gamma$ and $\delta$ are now easily obtained by simple algebra.
\end{proof}

A combination of the two facts now enables us to show that the generalized DR model $(\nu_{n})_{n\ge 0}$ with linear fractional input can be explicitly described by the evolution of a two-dimensional recursive sequence.

\begin{lemma}\label{p:pqab}
Let $\nu_{0}=\mathsf{LF}(\alpha, \beta)$ for some $\alpha, \beta>0$ with $\alpha+\beta\ge 1$ and ${\tt R} \eqdist \mathcal{G}(p)$ for some $p\in (0,1)$. Then $\nu_{n}= \mathsf{LF}(\alpha_{n}, \beta_{n})$ for any $n\ge 0$, where the $(\alpha_{n},\beta_{n})$ are given by the recursion $(\alpha_{0},\beta_{0})=(\alpha, \beta)$ and
\begin{equation}\label{eqn:pqab}
(\alpha_{n+1},\beta_{n+1})\,=\,\left(\frac{p\alpha_{n}}{d_{n}},\frac{1-p+p\beta_{n}}{d_{n}}\right)\quad
\text{with}\quad d_{n}\,=\,\varphi\left(\frac{1-p+p \beta_{n}}{1-p+p (\alpha_{n}+\beta_{n})}\right).
\end{equation}
\end{lemma}

\begin{proof}
This is a direct consequence of the Facts~\ref{f:1} and~\ref{f:2}.
\end{proof}

To transform the recursive equation \eqref{eqn:pqab} into \eqref{eqn:recursiveEquationReducedB}, which is the next step, we introduce the function
\begin{equation}\label{psi(x)}
\psi(x)\ :=\ \frac{1}{p}\,\varphi\left(\frac{x}{x+1}\right),\quad x\ge 0.
\end{equation}
which is nonnegative, increasing, $\mathcal{C}^{\infty}$ and bounded, with
\[ \psi(0)\ =\ \phi(0)/p\ =\ 0\ \text{ and }\ \lim_{x\to \infty} \psi(x)\ =\ \phi(1)/p\ =\ 1/p,\]
using that $\P(\mathtt{Z} = 0) = 0$.
In particular, there exists a unique $\xi>0$  such that
\begin{equation*}
\psi(\xi)=1.
\end{equation*}
We use this parameter to renormalize the recursion satisfied by $(\alpha_{n},\beta_{n})$, transforming it into the form given in \eqref{eqn:recursiveEquationReducedB}.

\begin{proposition}
\label{prop:reductionToRecursion1}
Let $(\alpha_{n},\beta_{n})_{n\ge 0}$ be a sequence satisfying \eqref{eqn:pqab}. For all $n\in \N$, we define the following reparametrization:
\begin{equation}\label{eqn:reparametrization}
u_{n}\,:=\,\psi'(\xi) \frac{1-p}{p \alpha_{n}}\quad\text{and}\quad v_{n}\,:=\,\psi'(\xi) \left(\frac{\beta_{n}}{\alpha_{n}}-\xi\right).
\end{equation}
Then $(u_{n},v_{n})_{n\ge 0}$ satisfies the recursion in \eqref{eqn:recursiveEquationReducedB}, with $\Psi$ defined by
\begin{equation}\label{Psi-1}
 \Psi(x)\,:=\,\psi\left( \frac{x}{\psi'(\xi)}+\xi\right)\quad\text{for } x\in I\,:=\,(- \psi'(\xi) \xi, \infty).
\end{equation}
\end{proposition}

Straightforward calculations provide that $\Psi(0)=\Psi'(0)=1$ (by construction) and also $\Psi(\infty)=p^{-1}$. Moreover, for all $\delta>0$, the restriction of $\Psi$ to the interval $[-\psi'(\xi)\xi+\delta,\infty)$ can be extended to a function that satisfies \eqref{ass:psi}. Since $(v_{n})_{n\ge 0}$ is nondecreasing, we can therefore apply general results on solutions of \eqref{eqn:recursiveEquationReducedB} to this generalized DR model.

\begin{proof}
Introducing the parameterization
$$ (x_{n},y_{n})\,:=\,\left(\frac{\beta_{n}}{\alpha_{n}},\frac{1-p}{p \alpha_{n}}\right)\quad\text{for }n\ge 0, $$
it follows that $x_{n+1}= x_{n}+y_{n}$ and
$$ \frac{1-p+p (\alpha_{n}+\beta_{n})}{1-p +p\beta_{n}}\ =\ 1+ \frac{ p \alpha_{n} }{1-p +p\beta_{n}}\ =\ 1+ \frac{1}{y_{n}+x_{n}}\ =\ 1+ \frac{1}{x_{n+1}} $$
for each $n\ge 0$. From this, we obtain
$$ \frac{y_{n+1}}{y_{n}}\ =\ \frac{\alpha_{n}}{a_{n+1}}\ =\ \frac1{p}\,\varphi\bigg(\frac{x_{n+1}}{1+x_{n+1}}\bigg) $$
and then the recursive equation
\begin{align*}
(x_{n+1}, y_{n+1})\ =\ (x_{n}+y_{n} ,   y_{n} \psi(x_{n+1})),
\end{align*}
valid for all $n\ge 0$.
Finally, the proof is concluded by noting that $v_{n}=\psi'(\xi)(x_{n}-\xi)$ and $u_{n}=\psi'(\xi) y_{n}$.
\end{proof}

Observe that the free energy of the generalized DR model with linear fractional input, starting from $\nu_{0}=\mathsf{LF}(\alpha, \beta)$, is defined by
\begin{align*}
  F_\mathsf{LF}(\alpha, \beta)\ &=\ \lim_{n\to\infty} \frac{p^n}{\alpha_{n}}\ =\ \frac{p}{(1-p) \psi'(\xi)}  \lim_{n\to\infty} \Psi(\infty)^n u_{n}\\
  &=\ \frac{p}{(1-p)\psi'(\xi)} F(u_0,v_0),
\end{align*}
with the notation of \eqref{eqn:freeEnergyb}. As a consequence, Theorems~\ref{thm:criticalCurve} and~\ref{thm:main} can be translated as follows.
\begin{theorem}\label{cor:h-LF}
Assuming $\E(\log \mathrm{Z})<\infty$, the following assertions hold:

\begin{enumerate}[(a)]
\item  There exists a unique nonincreasing, Lipschitz continuous function $h$ which satisfies
$$ h(x+h(x))\ =\ \Psi(x+h(x)) h(x)\quad\text{for all }x \in (-\psi'(\xi)\xi,0], $$
and $h(x) \sim x^2/2$ as $x\uparrow 0$. Furthermore, for any $\alpha, \beta>0$ with  $\alpha+\beta\ge 1$,
$$ h\left( \psi'(\xi)\left( \frac{\beta}{\alpha}-\xi \right)\right)\ <\ \psi'(\xi)\frac{1-p}{p\alpha}\quad\iff\quad F_\mathsf{LF}(\alpha, \beta)\,>\,0. $$

\item For $\alpha, \beta>0$ such that $\alpha +\beta\ge 1$ and $\beta/\alpha\le \xi$, let
$$ \gamma_{*}\, :=\, \inf\{\gamma \in (0, \alpha+\beta] : F_\mathsf{LF}(\alpha/\gamma,\beta/\gamma)>0\}. $$
Then
\begin{equation}
h\left( \psi'(\xi)\left( \frac{\beta}{\alpha}-\xi \right)\right)\ <\ \psi'(\xi)\frac{1-p}{p\alpha}(\alpha+\beta), \label{hyp-LF}\end{equation}
implies
$$ \gamma_{*}\ =\ \frac{ p\,\alpha \,h\left( \psi'(\xi)\left(\frac{\beta}{\alpha}-\xi \right)\right)}{\psi'(\xi) (1-p)}\ \in\ [0,\alpha+\beta), $$
and there exists $C_{\alpha,\beta}>0$ such that
$$ \lim_{\gamma \downarrow \gamma_{*}} (\gamma-\gamma_{*})^{1/2} \log F_\mathsf{LF}(\alpha/\gamma,\beta/\gamma)\ =\ -C_{\alpha,\beta}. $$
\end{enumerate}
\end{theorem}

Note that if \eqref{hyp-LF} fails, then $F_\mathsf{LF}(\alpha/\gamma,\beta/\gamma)=0$ for any $\gamma\in (0, \alpha+\beta]$.

\begin{proof}
By a standard Tauberian argument, $\E(\log\mathtt{Z})<\infty$ can be used to obtain asymptotic estimates of the probability generating function $\varphi$ of $\mathtt{Z}$ about $1$. In particular, it implies that condition \eqref{ass:technical} holds (see \eqref{psi(x)} and \eqref{Psi-1} for the connection between $\varphi$ and the function $\Psi$ appearing in this condition). We can therefore apply Theorem~\ref{thm:criticalCurve} and Proposition~\ref{prop:freeEnergy} to infer (a). For part (b), we apply Theorem~\ref{thm:main}, observing that with our notation, we have
$$ u_{0}\,=\,\psi'(\xi)\frac{1-p}{p\alpha}\, \gamma \quad \text{and} \quad v_{0}\,=\,\psi'(\xi)\left(\frac{\beta}{\alpha}-\xi\right), $$
i.e., $u_{0}$ varies with $\gamma$. The assertions now follow by immediate translation.
\end{proof}

The following result about the critical regime follows in a similar manner by a translation of Theorem~\ref{thm:criticalEvolution}.

\begin{corollary}\label{cor:LF-critical}
Let $\alpha, \beta>0$ be such that $\alpha +\beta\ge 1$ and $\beta/\alpha\le \xi$, and assume \eqref{hyp-LF}.  Let $(X_{n})_{n\ge 0}$ be a sequence of random variables such that $X_{n}$ has law $\nu_{n}$ for each $n$, where $\nu_{0}= \mathsf{LF}(\alpha/\gamma_{*},\beta/\gamma_{*})$. Then
$$ \P(X_{n}\ge 1)\ \sim\ \frac{2p}{(1-p) \psi'(\xi)(1+\xi)} \frac1{n^2},\quad\text{as }n\to\infty $$
and $X_{n}$, conditioned on $\{X_{n}\ge 1\}$, converges in law to $\mathcal{G}(\frac{1}{1+\xi})=\mathsf{LF}(\frac1{1+\xi}, \frac{\xi}{1+\xi})$.
\end{corollary}

\subsection{Generalized DR model on {$\R_+$} with continuous linear fractional input}
\label{subsec:drContinu}

The \textit{solvable} generalized DR model on $\R_+$, which will be studied in this section, can be viewed as the continuous counterpart to the model with linear fractional input. It is based on the assumption that $\mathtt{Z}$ is a generic $\R_+$-valued random variable, with $\nu_{0}$ taken as a mixture of an exponential distribution and a Dirac mass at 0. This mixture is referred to as a continuous linear fractional distribution in \cite{AlsmeyerHung:25}, as it corresponds to a continuous measure with a linear fractional Laplace transform (as opposed to the probability generating function used in the discrete case). For the sake of symmetry with the previous section, let us define these random variables formally.

\begin{definition}[Continuous linear fractional distributions]
Let $\lambda>0$, $\varrho \in [0,1]$ and $X$ be a random variable. We say that $X$ has a continuous linear fractional distribution with parameters $\lambda,\varrho$ and write $X\eqdist\mathsf{CLF}(\lambda,\varrho)$ if $X$ takes values in $\R_+$ and
\begin{equation}\label{eqn:tdf}
\P(X>x)\,=\,\varrho e^{-\lambda x}\quad\text{for all }x>0.
\end{equation}
This implies $\P(X=0)=1-\varrho$, $\E(X)=\varrho/\lambda$, and
\begin{equation}\label{eqn:lapTransform}
\E\left( e^{-\mu X} \right)\,=\,1-\varrho+\frac{\varrho \lambda}{\lambda+\mu}\quad\text{for all }\mu>-\lambda.
\end{equation}

\end{definition}

Note that continuous linear fractional distributions actually form the unique two-parameter family of probability measures on $\R_+$ whose Laplace transforms are linear fractional. Therefore, it is not surprising that, similar to the integer-valued linear fractional distributions, these continuous distributions define a solvable family within the context of the generalized DR model.

We also remark that this family corresponds to the family of probability distributions introduced in \cite{HMP} for the continuous-time DR model. There does not seem to be a direct link between the continuous- and the discrete-time models. However, it has been proved in \cite{lz-asymp} that one can obtain the continuous-time DR model as a scaling limit of the discrete-time model.

Similar to the previous section, we begin by assembling a couple of facts about the laws $\mathsf{CLF}(\lambda,\varrho)$.

\begin{fact}
\label{f:1m}
Let $\varrho\in [0,1]$, $\lambda>0$, ${\tt R} \eqdist \mathcal{G}(p)$ for some $p\in (0,1)$ and $(X_{n},n\geq 1)$ be a sequence of i.i.d.~random variables with common law $\mathsf{CLF}(\lambda,\varrho)$ and independent of ${\tt R}$. Then
$$ \sum_{j=1}^{\tt R}X_{j}\ \eqdist\ \mathsf{CLF}\left(\frac{\lambda p}{p+(1-p)\varrho},\frac{\varrho}{p+(1-p)\varrho}\right). $$
\end{fact}

The proof of this fact follows by computing the Laplace transform of the randomized sum and its identification using \eqref{eqn:lapTransform}. The second fact to notice is the following counterpart of Fact~\ref{f:2}.

\begin{fact}
\label{f:2m} Let $\varrho\in [0, 1]$ and $\lambda>0$.
If $X\eqdist\mathsf{CLF}(\lambda,\varrho)$ and ${\tt Z} $ is an independent random variable taking values in $\R_+$ and with Laplace transform $\varphi$, then
$$ (X-{\tt Z} )_+\ \eqdist\ \mathsf{CLF}\left(\lambda,\varrho\,\varphi(\lambda)\right). $$
\end{fact}

In view of \eqref{eqn:tdf}, it suffices to note that $\P((X-{\tt Z})_+>x)=\P(X>x+{\tt Z})=\rho\,\E(e^{-\lambda(x+{\tt Z})})$ $=\rho\,\varphi(\lambda)e^{-\lambda x}$ for all $x>0$.

As a consequence of these two facts, we directly infer that, if $\nu_{0} =\mathsf{CLF}(\lambda_{0},\varrho_{0})$, then the $\nu_{n}$ in the generalized DR model defined by \eqref{eqn:derridaRetauxGeneralized} are all continuous linear fractional distributions, namely
\begin{equation}\label{eqn:formulaLambdaRho}
\nu_{n} =\mathsf{CLF}(\lambda_{n},\varrho_{n}).
\end{equation}
for suitable recursively defined $(\lambda_{n}, \varrho_{n})\in (0,\infty) \times [0,1]$.
However, we need to introduce a new parametrization of $\mathsf{CLF}(\lambda,\varrho)$ such that the sequence $(\nu_{n})_{n\ge 0}$ under this new parametrization satisfies \eqref{eqn:recursiveEquationReducedB}.

Let us define the function
\begin{equation}
\gamma(\theta)\,:=\,\frac1{p} \E(e^{- {\tt Z}/\theta})\quad\text{for }\theta>0, \label{gamma}
\end{equation}
and $\tau>0$ as the unique number satisfying
\begin{equation}
\gamma(\tau)\,=\,1. \label{tau}
\end{equation}
Moreover, we let $\Psi$ be given by
\begin{equation}
\Psi(x)\,:=\, \gamma\left(\frac{x}{\gamma'(\tau)}+\tau\right)\quad\text{for }x \in I:= (-\tau \gamma'(\tau),\infty), \label{Psi-2}
\end{equation}
and note that this function satisfies $\Psi(0)=\Psi'(0)=1$.

\begin{proposition}
\label{prop:reductionToRecursion2}
Let $(\varrho_{n},\lambda_{n})_{n\ge 0}$ be the sequence determined by \eqref{eqn:formulaLambdaRho} and define
$$ u_{n}\,:=\,\gamma'(\tau)\,\frac{1-p}{p}\,\frac{\varrho_{n}}{\lambda_{n}}\quad\text{and} \quad v_{n}\,:=\,\gamma'(\tau)\left( \frac{1}{\lambda_{n}}-\tau \right)\quad\text{for }n\ge 0. $$
Then $(u_{n},v_{n})_{n\ge 0}$ satisfies \eqref{eqn:recursiveEquationReducedB} with the function $\Psi$ in \eqref{Psi-2} provided that $v_{0}>-\tau \gamma'(\tau)$.
\end{proposition}

We note that for all $\delta>0$, the restriction of $\Psi$ to $[-\psi'(\xi)\xi+\delta,\infty)$ can be extended to a bounded $\mathcal{C}^2$ function on $\R$. Using that $(v_{n})_{n\ge 0}$ is nondecreasing, we can once again apply general results on solutions to \eqref{eqn:recursiveEquationReducedB} to this generalized DR model. The proof of Proposition~\ref{prop:reductionToRecursion2} involves straightforward computations, similar to those for Proposition~\ref{prop:reductionToRecursion1}, and we therefore omit the details.

As in the previous subsection, we can compute the free energy of the generalized DR model starting from $\nu_{0} =\mathsf{CLF}(\lambda,\varrho)$, defined by
$$ F_{\mathsf{CLF}}(\lambda,\varrho)\ =\ \lim_{n\to\infty} \frac{p^n\varrho_{n}}{\lambda_{n}} \ =\ \frac{p}{(1-p)\gamma'(\tau)}  \lim_{n\to\infty} p^n u_{n}. $$
Therefore, Theorems~\ref{thm:criticalCurve} and~\ref{thm:main} can once again be applied to this model, yielding the following result.

\begin{theorem}\label{cor:h-exp}
Assuming $\E(\log \mathrm{Z})<\infty$, the following assertions hold:
\begin{enumerate}[(a)]
\item There exists a unique nontrivial function $h$ satisfying
$$ h(x+h(x))=\Psi(x+h(x)) h(x)\quad\text{for all }x \in (-\psi'(\xi)\xi,0]. $$
Furthermore, for any $\varrho \in [0,1]$ and $\lambda>0$,
$$ h\left( \gamma'(\tau)\left( \frac{1}{\lambda}-\tau \right) \right)\ <\ \gamma'(\tau)\,\frac{1-p}{p}\,\frac{\varrho}{\lambda}\quad\iff\quad F_{\mathsf{CLF}}(\lambda,\varrho)\,>\,0. $$
\item For $\lambda\ge1/\tau$, let
$$ \varrho_{\lambda}^{*}\,:=\,\inf \{\varrho \in (0,1]:F_\mathsf{CLF}(\lambda,\varrho)>0\}. $$
Then
\begin{equation}
\lambda\,p\,h\left( \gamma'(\tau)\left( \frac{1}{\lambda}-\tau \right) \right)\,<\,\gamma'(\tau) (1-p), \label{hyp-exp}
\end{equation}
implies
$$ \varrho_{\lambda}^{*}\ =\ \frac{\lambda p}{\gamma'(\tau) (1-p)}  h\left( \gamma'(\tau)\left( \frac{1}{\lambda}-\tau \right) \right), $$
and there exists $C_\lambda>0$ such that
$$ \lim_{\epsilon\to 0} \epsilon^{1/2} \log F(\varrho_{\lambda}^{*}+\epsilon,\lambda)\,=\,-C_\lambda. $$
\end{enumerate}
\end{theorem}

Note that if  \eqref{hyp-exp} fails, then $F_\mathsf{CLF}(\lambda,\varrho)=0$.

Once again, we can apply Theorem~\ref{thm:criticalEvolution} to obtain the following counterpart of  Corollary~\ref{cor:LF-critical} in the present situation when $\nu_{0}$ lies on the critical curve.

\begin{corollary}\label{cor:CLF-critical}
Let $\lambda\ge 1/\tau$ and assume \eqref{hyp-exp}.  Let $(X_{n})_{n\ge 0}$ be a sequence of random variables such that $X_{n}$ has law $\nu_{n}$ for each $n$, where $\nu_{0}=\mathsf{CLF}(\varrho_{\lambda}^{*},\lambda)$. Then $X_{n}$, conditioned on $\{X_{n}>0\}$, converges in law to the exponential distribution with parameter $1/\tau$, i.e., $\mathsf{CLF}(1/\tau,1)$.
\end{corollary}

\section{Simple properties of the Derrida-Retaux recursion}
\label{sec:basics}

In this section, we present some straightforward properties of the solutions to the recursive equation \eqref{eqn:recursiveEquationReducedB}. Let $\Psi$ be a bounded, nonnegative, nondecreasing $\mathcal{C}^2$ function such that $\Psi(0)=\Psi'(0)=1$, i.e.\ that $\Psi$ satisfies \eqref{ass:psi}. We begin by proving \eqref{eqn:limits} through a characterization of the limits of a sequence $(u_{n},v_{n})_{n\ge 0}$ verifying \eqref{eqn:recursiveEquationReducedB}.

\begin{lemma}\label{lem:limits}
Given a  solution $(u_{n},v_{n})_{n\ge 0}$ to \eqref{eqn:recursiveEquationReducedB}, the following dichotomy holds: Either
$$ \lim_{n\to\infty} v_{n}\,=\,\infty\quad\text{and}\quad\lim_{n\to\infty}\frac{1}{n}\log u_{n}\,=\,\log \Psi(\infty)\,>\,0 $$
or
$$ \lim_{n\to\infty} v_{n}\,\leq\,0\quad\text{and}\quad\lim_{n\to\infty} u_{n}\,=\,0. $$
\end{lemma}

\begin{proof}
We remark that $v_{n+1}-v_{n}\,=\,u_{n}\ge 0$ for all $n\in\N$, thus $(v_{n})_{n\ge 0}$ is nondecreasing. As a result, $(v_{n})_{n\ge 0}$ either converges to a nonpositive limit, or all $v_{n}$ are nonnegative for sufficiently large $n$.

In the first situation, since $v_{n}$ converges to a finite limit, we have $u_{n}=v_{n+1}-v_{n}\to 0$ as $n\to\infty$, which is the desired conclusion.

Let us now assume that there exists $N\in\N$ such that $v_{n}>0$ for all $n \ge N$. In this case, we infer $u_{n+1}=u_{n}\Psi(v_{n+1})>u_{n}$ for all $n\ge N$, using that $\Psi'(0)=1$. Therefore, both $u_{n}$ and $v_{n}$ are strictly increasing for $n\ge N$ and, as a consequence, we have $v_{n}-v_{N}\ge (n-N)\,u_{N}$, which shows that $\lim_{n\to\infty} v_{n}=\infty$. With this, we conclude
$$ \lim_{n\to\infty} \frac{u_{n+1}}{u_{n}}\ =\ \lim_{n\to\infty} \Psi(v_{n+1})\ =\ \Psi(\infty) $$
and then
$$ \lim_{n\to\infty} \frac{1}{n} \log u_{n}\ =\ \lim_{n\to\infty} \frac{1}{n} \sum_{k=0}^{n-1}\log \frac{u_{k+1}}{u_{k}}\ =\ \log \Psi(\infty), $$ 
by an appeal to the Stolz-Ces\`aro lemma.
\end{proof}

Next, we prove that the free energy is well-defined and characterizes the supercritical domain when $\Psi$ satisfies \eqref{ass:technical}, thereby proving Proposition~\ref{prop:freeEnergy}.

\begin{proof}[Proof of Proposition~\ref{prop:freeEnergy}]
Recall that the free energy is defined as
$$ F(u_{0},v_{0})=\liminf_{n\to\infty} \Psi(\infty)^{-n} u_{n} $$
and note that, if $\lim_{n\to\infty} v_{n}\le  0$, then $F(u_{0},v_{0})=0$ by Lemma~\ref{lem:limits}. Therefore, we need only consider the case where $v_{n}\to \infty$.

The monotonicity of $\Psi$ implies
$$  \frac{u_{n+1}}{\Psi(\infty)^{n+1}}\ =\ \frac{u_{n}}{\Psi(\infty)^n} \frac{\Psi(v_{n+1})}{\Psi(\infty)}\ \le \ \frac{u_{n}}{\Psi(\infty)^n}, $$
and since $(\Psi(\infty)^{-n}u_{n})_{n\ge 0}$ is nonincreasing, we infer that $F(u_{0},v_{0})$ is well-defined as the limit of this sequence.

We now assume that, in addition, \eqref{ass:technical} holds and demonstrate that $F(u_{0},v_{0})>0$ whenever $v_{n}\to \infty$. By Lemma~\ref{lem:limits}, for any $1<\varrho<\Psi(\infty)$, we have $u_{n}>\varrho^n$ for all $n$ large enough. Thus, since $v_{n}=v_{0}+\sum_{i=0}^{n-1} u_i \geq v_0 + u_{n-1}$, we have
\[ \liminf_{n\to\infty} \varrho^{-n}v_{n}\ \ge\ \liminf_{n\to\infty}\varrho^{-n} u_{n-1}\ \ge\ \rho^{-1}. \]
Next, we use \eqref{eqn:recursiveEquationReducedB} to write
\begin{equation}
  \label{F:sum1}
  \frac{u_{n}}{\Psi(\infty)^n}\ =\ u_{0} \prod_{j=1}^n \frac{\Psi(v_{j})}{\Psi(\infty)}\ =\ u_{0} \exp\left( \sum_{j=1}^n \log \frac{\Psi(v_{j})}{\Psi(\infty)} \right)
\end{equation}
and then conclude from \eqref{ass:technical} that $\sum_{n\in\N} \log \frac{\Psi(v_{n})}{\Psi(\infty)}$ converges and thus that $u_{n}/\Psi(\infty)^n$ converges to a positive limit as $n\to\infty$.
\end{proof}

The following monotonicity lemma, which will frequently be used in our analysis, can be succinctly summarized by stating that \eqref{eqn:recursiveEquationReducedB} preserves the order when replacing $\Psi$ by a smaller or larger function.

\begin{lemma}
\label{lem:monotonicity}
Let $(u_{n},v_{n})_{n\ge 0}$ be a solution to \eqref{eqn:recursiveEquationReducedB} and $\underline{\Psi}, \bar{\Psi}$ be two nonnegative nondecreasing functions such that
$$ \underline{\Psi}(x)\,\le\,  \Psi(x)\,\le\,\bar{\Psi}(x)\quad\text{for all }x\in\R. $$
If $(\underline{u}_{n},\underline{v}_{n})_{n\ge 0}$ and $(\bar{u}_{n},\bar{v}_{n})_{n\ge 0}$ are solutions to \eqref{eqn:recursiveEquationReducedB} with $\Psi$ replaced by $\underline{\Psi}$ and $\bar{\Psi}$, respec\-tively, and
$0\le\underline{u}_{0}\le u_{0}\le\bar{u}_{0}$ and $\underline{v}_{0}\le v_{0}\le  \bar{v}_{0}$, then
\begin{equation}\label{eqn:monotony}
0\,\le\,\underline{u}_{n}\,\le\, u_{n}\,\le\,\bar{u}_{n} \quad \text{and}\quad\underline{v}_{n}\,\le\,   v_{n}\,\le\,\bar{v}_{n}
\end{equation}
for all $n\in\N$.
\end{lemma}

Note that we do not require $v_0$ to be positive in the above lemma.

\begin{proof}
The proof follows by a simple induction. Assuming that \eqref{eqn:monotony} holds for some $n\in\N$, we obtain by summation
$$ \underline{v}_{n+1}\,\le\,v_{n+1}\,\le\,\bar{v}_{n+1}. $$
Then, by using that $\underline{u}_{n}$ and $\underline{\Psi}(v_{n+1})$ are nonnegative, we find by immediate comparison that
\[ \underline{u}_{n+1}\,\le \,u_{n+1}\,\le\,\bar{u}_{n+1}. \qedhere \]
\end{proof}

We now examine a duality relationship for the recursion equation \eqref{eqn:recursiveEquationReducedB}, which is based on time--reversal. Specifically, we show that the type of the recursion remains unchanged when time is reversed.

\begin{proposition}\label{prop:duality}
Let $(u_{n},v_{n})_{n\ge 0}$ be a sequence defined recursively by \eqref{eqn:recursiveEquationReducedB}. We fix $N\in\N$ and define
$$ \check{u}_{n}\,:=\,u_{N-n} \quad \text{and} \quad \check{v}_{n}\,:=\,-v_{N-n+1}\quad\text{for }0\le n\le N. $$
Then, for all $0\le n<N$, we have
\begin{equation}\label{eqn:backwardRecursiveEquation}
\begin{pmatrix} \check{u}_{n+1}\\ \check{v}_{n+1}\end{pmatrix}\ =\ \begin{pmatrix} \check{u}_{n}/\Psi(-\check{v}_{n+1})\\ \check{u}_{n}+\check{v}_{n}\end{pmatrix}.
\end{equation}
\end{proposition}

We see from Proposition~\ref{prop:duality} that the backward evolution of $(u_{n},v_{n})_{n\ge 0}$ is a solution to \eqref{eqn:recursiveEquationReducedB} with $\Psi$ replaced by $\check{\Psi}(x)=1/\Psi(-x)$. We note that if $\Psi$ satisfies \eqref{ass:psi} and if $\lim_{x \to-\infty} \Psi(x)>0$, then $\check{\Psi}$ satisfies \eqref{ass:psi} as well and satisfies $\lim_{x\to -\infty}\check{\Psi}(x)>0$.

\begin{proof}
The proof follows by simple computations and can be omitted.
\end{proof}

\begin{remark}
We were not able to specify a law for the random variable ${\tt Z}$, neither discrete nor continuous, for which the time-reversal of the dynamics associated with the generalized Derrida--Retaux model can be interpreted as a dual Derrida--Retaux model. In particular, this time-reversal induces a slight mixing of time indices, associating $u_k$ with $v_{k+1}$, which partially explains this lack of interpretability. Time-reversal was a crucial tool in \cite{HMP} to describe the law of the critical tree associated with the Derrida--Retaux model; however, in that work the backward construction must be interpreted as a growth-fragmentation-type branching process.
\end{remark}

\section{Evolution along the critical line}
\label{sec:criticalCurve}

The main goal of this section is to describe the set $\mathcal{C}$, defined in \eqref{eqn:decompositionOfSpace}, as the set of initial conditions $(u_{0},v_{0})$ such that  $\lim_{n\to\infty} v_{n}=0$. As noted in Lemma~\ref{lem:limits}, the sequence $(v_{n})_{n\ge 0}$ either converges to a nonpositive limit or diverges to $\infty$. Moreover, by Lemma~\ref{lem:monotonicity}, we observe that the function $\phi_{v}$, which assigns to each $u \in\R_+$ the quantity $\lim_{n\to\infty} v_{n}$ with the initial conditions $(u_{0},v_{0})=(u,v)$,  is nondecreasing. The range of $\phi_{v}$ is $(-\infty,0] \cup \{\infty\}$.

We show that $\mathcal{C}$ is the graph of a continuous function, which can be described as follows:
\begin{equation}\label{h-general}
h : v \mapsto \inf\{u \in (0,\infty) : \phi_{v}(u)=\infty\}.
\end{equation}
Note that this definition of $h$ coincides with the one given in \eqref{eqn:defineCriticalCurve}.
Since $\phi_{v}$ is nondecreasing, we immediately deduce that $\lim_{n\to\infty} v_{n}=\infty$ if $u_0 < h(v_{0})$ and $\lim_{n\to\infty} v_{n}\le  0$ if $u_0 > h(v_{0})$. The behavior of the limit when $u_{0}=h(v_{0})$ remains unclear at this stage. In other words, we have:
\begin{equation}\label{eqn:firstObservation}
\{(u,v) : u<h(v)\}\subset\mathcal{C} \cup \mathcal{U} \quad \text{and} \quad \{(u,v) : u>h(v)\}\subset\mathcal{P},
\end{equation}
which is a first step toward proving \eqref{eqn:caracterisationOfPCU}. Finally, using Lemma~\ref{lem:monotonicity} again, we observe that $h$ is nonincreasing. Furthermore, since $(u_{n})_{n\ge 0}$ is nondecreasing when $v_{0}\ge 0$, we see that $h(x)=0$ for all $x\ge 0$.

The rest of the section is organized as follows. We prove in Subsection~\ref{subsec:hDefinition} that the function $h$ forms the only nontrivial solution to the functional equation \eqref{eqn:functionalEquationCriticalCurve} stated in Theorem~\ref{thm:criticalCurve}. This equation then allows us to identify the subcritical, critical and supercritical domains of the dynamics via \eqref{eqn:caracterisationOfPCU}. We then study the regularity of this function $h$ and finally, in Subsection~\ref{subsec:criticalEvolution}, the asymptotic behaviour of $(u_{n},v_{n})$ provided that $u_{0}=h(v_{0})$.

\subsection{Functional equation and analysis of the critical curve}
\label{subsec:hDefinition}

The main result of the section is the following proposition, which establishes the existence of a unique function $h$ satisfying \eqref{eqn:functionalEquationCriticalCurve}. Additionally, we show that the function $h$ corresponds to \eqref{h-general}.

\begin{proposition}
\label{p:g}
Let $A>0$ and $\Psi:[-A,\infty)\mapsto \R_+$ be a nondecreasing $\mathcal{C}^2$ function, such that $\Psi(0)=\Psi'(0)=1$. There exists a unique function  $g: [-A, \infty) \mapsto \R$  such that $g(x)=x$ for all $x\ge 0$, $g(x)>x$ for all $x\in [-A, 0)$ and
\begin{equation} \label{gxN}
g(g(x))\ =\ g(x)+\Psi(g(x)) (g(x)-x)\quad\text{for }x \in [-A, 0].
\end{equation}
Furthermore,
\begin{enumerate}[(i)]
\item the function $g$ is nondecreasing, $1$-Lipschitz -- that is, Lipschitz continuous  with Lip\-schitz constant 1 -- and satisfies
\begin{equation}\label{gtaylor}
g(x)- x\ \sim\ \frac{x^2}{2}\quad\text{as } x\uparrow 0,
\end{equation}
\item given a solution $(u_{n},v_{n})_{n\ge 0}$ to \eqref{eqn:recursiveEquationReducedB} with $v_0 \geq -A$, we have
\begin{equation}
\label{eqn:phaseSeparation}
\begin{cases}
\lim_{n\to\infty} v_{n}=\,\infty & \text{if } g(v_{0})-v_{0}\,>\,u_{0},\\
\lim_{n\to\infty} v_{n}=\,0 & \text{if } g(v_{0})-v_{0}\,=\,u_{0},\\
\lim_{n\to\infty} v_{n}<\,0 &\text{if } g(v_{0})-v_{0}\,<\,u_{0}.
\end{cases}
\end{equation}
\end{enumerate}
\end{proposition}

\begin{proof}
We observe that the fact that $x \mapsto g(x)-x$ satisfies \eqref{eqn:phaseSeparation} implies $g(x)=x+h(x)$ for all $x\in\R$, with $h$ the function defined in \eqref{h-general}. Consequently, demonstrating that a solution to \eqref{gxN} satisfies \eqref{eqn:phaseSeparation} establishes the uniqueness of the function.

In the first part of the proof, we show the existence of a function $g$ satisfying \eqref{gxN} via an approximation argument, using a fixed-point approach. Next, we establish the regularity of $g$ stated in part (i) by analyzing the properties of its approximating sequence. Finally, we prove that $g$ satisfies \eqref{eqn:phaseSeparation}.\vspace{.1cm}

{\it Proof of \eqref{gxN}.} We fix an arbitrary constant $K> 0$ such that
$$ K\ \ge\ \sup_{x \in [-A,0]} (\Psi(x)+(x+A) \Psi'(x)) $$
and introduce the auxiliary function
$$ \sigma :  [-A,0]\,\ni\,x\ \mapsto\ Kx -(x+A) \Psi(x). $$
It is straightforward to verify that $\sigma'(x)\ge 0$ for all $x \in [-A,0]$, implying that $\sigma$ is~non\-decreasing. We also note that \eqref{gxN} can be rewritten as
\begin{equation}\label{gxN3}
(K+1)g(x)\ =\ g(g(x))+\sigma(g(x))+(x+A) \Psi(g(x))\quad\text{for }x \in [-A,0].
\end{equation}

We now define recursively a sequence of functions $(g_{n})_{n\ge 1}$ on $[-A,0]$ as follows. We begin by fixing $g_1$ as the unique solution to the differential equation $y'=\Psi(y)$ on $[-A,0]$ with the initial condition $y(0)=0$. Then for any $n\ge 1$ and $x \in [-A,0]$, we define $g_{n+1}$ by the relation
\begin{align}
(K+1) g_{n+1}(x) :=\ &g_{n}(g_{n} (x))\,+\,\sigma(g_{n}(x))\,+\,(x +A) \Psi(g_{n}(x)) \label{gn+1}\\
=\ &g_{n}(g_{n} (x))\,+\,K g_{n}(x)\,-\,(g_{n}(x)-x)\Psi(g_{n}(x)). \label{gn+1'}
\end{align}
In the second line, we used the definition of $\sigma$. We prove by induction on $n$ that the sequence $(g_{n})_{n\ge 1}$ is nondecreasing on $[-A,0]$ and consists of nondecreasing functions that are $\mathcal{C}^2$, $1$-Lipschitz, and satisfy $g_{n}(0)=0$ and $g_{n}'(0) =1$ for all $n$.

We immediately observe from the definition of $g_1$ that it is a $\mathcal{C}^2$ function with $g_1(0)=0$, using the fact that $\Psi$ is $\mathcal{C}^{1}$, and that $g_1'(0)=\Psi(0)=1$. Moreover, we have
$$ 0\,\le\,\Psi(g_1(x))\,=\,g_1'(x)\,\le\,1\quad\text{for all }x \in [-A,0]. $$
Therefore $g_1$ is nondecreasing and $1$-Lipschitz, and since $g_1'=\Psi \circ g_1$ is also nondecreasing, we have that $g_1$ is convex. As a result, by the definition in \eqref{gn+1'},
$$ (K+1) (g_2(x)-g_1(x))\ =\ g_1(g_1(x))-\left( g_1(x)+(g_1(x)-x)\Psi(g_1(x))\right)\ \ge\  0 $$
for all $x \in [-A,0]$, where we used the fact that $y \mapsto (y-x) g_1'(x)+g_1(x)$ is the tangent of $g_1$ at point $x$, and hence smaller than $g_1$, in particular at $y=g_1(x)$.

Turning to the inductive step, we fix $n\in\N$ and assume that $g_{n}$ is a nondecreasing function that is $\mathcal{C}^2$ and $1$-Lipschitz with $g_{n}(0)=0$ and $g_{n}'(0)=1$. We also assume that $g_{n+1}(x)\ge g_{n}(x)$ for all $x \in [-A,0]$. By \eqref{gn+1'}, $g_{n+1}$ is then clearly $\mathcal{C}^2$ as well, and we have
$$ (K+1)g_{n+1}(0)\ =\ g_{n}(g_{n}(0))+K g_{n}(0)-g_{n}(0) \Psi(g_{n}(0))\ =\ 0 $$
since $g_{n}(0)=0$.

Regarding the first derivative of $g_{n+1}$, we compute using \eqref{gn+1'}
\begin{multline*}
(K+1)g_{n+1}'(x)\\
\begin{split}
&=\ g_{n}'(g_{n}(x)) g_{n}'(x)+K g_{n}'(x)-(g_{n}'(x)-1)\Psi(g_{n}(x))-(g_{n}(x)-x) \Psi'(g_{n}(x)) g_{n}'(x)\\
&= g_{n}'(g_{n}(x)) g_{n}'(x)+\Psi(g_{n}(x))+g_{n}'(x)\big(K-\Psi(g_{n}(x)) -(g_{n}(x)-x) \Psi'(g_{n}(x))\big).
\end{split}
\end{multline*}
We first observe that $g_{n+1}'(0)=1$. Moreover, since $g_{n}$ is nondecreasing and $1$-Lipschitz, we know that $g_{n}(x)-x\ge 0$ for all $x \in [-A,0]$ and can therefore estimate
$$ (K+1)g_{n+1}'(x)\ \le\ 1+\Psi(g_{n}(x))+\left( K-\Psi(g_{n}(x)) \right)\ \le\ K+1. $$
Additionally, since $g_{n}(x) \in [-A,0]$, we can use the definition of $K$ to obtain the following lower bound:
\begin{align}
\begin{split}\label{eqn:formula}
K-\Psi(g_{n}(x))&-(g_{n}(x)-x) \Psi'(g_{n}(x))\\
&\ge\ K-\sup_{y \in [-A,0]} (\Psi(y)+(y+A)\Psi'(y))\ \ge\ 0.
\end{split}
\end{align}
Thus, we conclude that $g_{n+1}'(x) \in [0,1]$ for all $x \in [-A,0]$, which shows that $g_{n+1}$ is nondecreasing and $1$-Lipschitz.

Finally, we show that $g_{n+2}\ge g_{n+1}$, using \eqref{gn+1}. We have
\begin{multline*}
(K+1)(g_{n+2}(x)-g_{n+1}(x))\ =\ g_{n+1}(g_{n+1}(x))-g_{n}(g_{n}(x))\\
+\ \sigma(g_{n+1}(x))-\sigma(g_{n}(x))+(x+A) (\Psi(g_{n+1}(x))-\Psi(g_{n}(x))).
\end{multline*}
Since $g_{n}$ and $g_{n+1}$ are nondecreasing with $g_{n+1}\ge g_{n}$, we have $g_{n+1}(g_{n+1}(x))-g_{n}(g_{n}(x))\ge 0$. Similarly, using the monotonicity of $\sigma$ and $\Psi$, and noting that $x+A\ge 0$, we infer that $g_{n+2}(x)\ge g_{n+1}(x)$ for all $x\in [-A,0]$.
  Next, we define $g$ on $[-A,\infty)$ by taking the increasing limit of $g_n$ on $[-A,0]$, i.e.
\[
  g(x) := \begin{cases}
     \lim_{n\to\infty} g_{n}(x) &\text{if }x \in [-A,0],\\
     x &\text{if } x > 0.
  \end{cases}
\]
Using the properties of the sequence $(g_{n})_{n\ge 1}$, we observe that $g$ is  nondecreasing, $1$-Lip\-schitz, and satisfies $g(0)=0$. Moreover, for all $x \in [-A,0)$, we have $g(x)\ge g_1(x)>x$. By continuity of $\Psi$, we also have
$$ (K+1)g(x)=g(g (x))+ K g(x)- (g(x)-x) \Psi(g(x))\quad\text{for all }x \in [-A,0], $$
which shows that $g$ satisfies equation \eqref{gxN}.\vspace{.1cm}

{\it Proof of \eqref{gtaylor}.} It remains to show that $g(x)-x\sim x^2/2$ as $x\to 0$ on $[-A,0]$.  First, observe that by construction, $g\ge g_1$, where $g_1$ is a $\mathcal{C}^2$ function with $g_1(0)=0$, and $g_1'(0)=g_1''(0)=1$. Therefore, by a Taylor expansion of $g_1$ around $x=0$, we have
$$ g(x)-x\ \ge\ g_1(x)-x\ =\ \frac{x^2}{2}(1+o(1)). $$
To complete the proof, we therefore only need to show that for any $w>1/2$, there exists $\delta>0$ such that
$$ g(x)\,\le\,x+w x^2\quad\text{for all }x \in [-\delta,0]. $$
To this end, we fix $\delta \in (0,A)$ such that $g_1(x)\le  x+wx^2$  and $\Psi(x)\ge (1+wx)^2$ for all $x \in [-\delta,0]$, using the fact that $\Psi(0)=\Psi'(0)=1$. We will prove by induction that for all $n\ge 1$ and $x \in [-\delta,0]$, the inequality $g_{n}(x)\le Q(x) := x+wx^2$ holds, from which \eqref{gtaylor} will follow by passage to the limit.

Assuming $g_{n}(x)\le  Q(x)$ for all $x \in [-\delta,0]$ and using formula \eqref{gn+1}, it follows that
\begin{align*}
(K+1)g_{n+1}(x)\ &\le\ Q(Q(x))+\sigma(Q(x))+(x+ A)\Psi(Q(x))\\
&\le\ Q(Q(x))+KQ(x)-(Q(x) -x) \Psi(Q(x)),
\end{align*}
since $g_{n}$ is $1$-Lipschitz, thus $g_{n}(x)\ge-\delta$ for all $x \in[-\delta,0]$. Now consider the expression $Q(Q(x))-(Q(x) -x)\Psi(Q(x))$:
\begin{gather*}
Q(Q(x))-(Q(x)-x) \Psi(Q(x))\ =\ Q(x)+w Q(x)^2-w x^2 \Psi(Q(x)).
\intertext{This simplifies to}
Q(x)+w x^2\left((1+wx)^2-\Psi(Q(x))\right).
\end{gather*}
Since $Q(x)\ge x$, we have $(1+wx)^2-\Psi(Q(x))\le  0$, which shows that
$$ Q(Q(x)) -(Q(x) -x) \Psi(Q(x))\ \le\ Q(x). $$
Hence, we conclude that
$$ (K+1)g_{n+1}(x)\ \le\ (K+1) Q(x), $$
which completes the induction step and the proof of \eqref{gtaylor}.

{\it Proof of \eqref{eqn:phaseSeparation}.} To complete the proof of Proposition~\ref{p:g}, it remains to show that the function $h^{*}: x \mapsto g(x)-x$ and the function $h$ in \eqref{h-general} are identical on  $[-A, \infty)$.  Using the properties of $g$,  we observe that $h^{*}$ is a continuous solution of the functional equation
\begin{equation}\label{h}
h^{*}(x+h^{*}(x))\ =\ h^{*}(x) \Psi(x+h^{*}(x))\quad\text{for all }x \in [-A,0].
\end{equation}
such that $0\le h^{*}(x) \le (-x)_+$ for all $x\in \R$ and $h(x) >0$ for $x<0$. We underscore that to prove $h^{*}$ coincides with $h$ on $[-A, \infty)$, no additional regularity conditions on $h^{*}$, such as monotonicity or $1$-Lipschitz continuity, are required.

Recall that $h(x)=0$ for all $x\ge 0$. Hence $h^*=h$ on $\R_+$. It remains to show that $h^*(v_0)=h(v_0)$ for any $v_{0} \in [-A,0)$. To this end,
let  $(u_{n}^{*},v_{n}^{*})_{n\ge 0}$ denote the solution of the recursive equation \eqref{eqn:recursiveEquationReducedB} with initial conditions $(u_{0}^{*},v_{0}^{*})=(h^{*}(v_{0}),v_{0})$. By induction, we immediately obtain $u_{n}^{*}=h^{*}(v_{n}^{*})$ for all $n\ge 1$. This follows from the recurrence relation
$$ u^{*}_{n+1}\ =\ u^{*}_{n} \Psi(v^{*}_{n+1})\ =\ h^{*}(v^{*}_{n}) \Psi(h^{*}(v^{*}_{n})+v^{*}_{n})\ =\ h^{*}(v^{*}_{n}+h^{*}(v^{*}_{n}))\ =\ h^{*}(v^{*}_{n+1}), $$
  when using \eqref{h} for the last equality.  Note that if $v^{*}_{n}\le 0$, then  $v^{*}_{n+1}= v_{n}^{*}+ h^{*}(v_{n}^{*}) \le 0$ as well since $h^{*}(x) \le (-x)_+$. Therefore, $\sup_{n\ge 0} v^{*}_{n}\le  0$. Since $(v_{n}^{*})_{n\ge 0}$ is a nondecreasing sequence, we have $v^{*}_{\infty}:=\lim_{n\to\infty} v^{*}_{n}\le 0$. Furthermore, $u^{*}_{\infty}:= \lim_{n\to\infty}u^{*}_{n}=h^{*}(v^{*}_{\infty})$, and thus $u^{*}_{\infty}>0$ if $v^{*}_{\infty}<0$.  But this contradicts \eqref{eqn:limits}, and we conclude $\lim_{n\to\infty} v^{*}_{n}=0$.

Next, let $(u_{n},v_{n})_{n\ge 0}$ be a solution to \eqref{eqn:recursiveEquationReducedB} with $u_{0}>h^{*}(v_{0})$. By Lemma~\ref{lem:monotonicity}, we know that $u_{n}\ge u^{*}_{n}$ and $v_{n}\ge v^{*}_{n}$ for all $n\in\N$. Moreover, we have the inequality
$$ v_{n+1}-v^{*}_{n+1}\ =\ v_{n}-v^{*}_{n}+u_{n}-u_{n}^{*}\ \ge\ v_{n}-v_{n}^{*}\ \ge\ v_1-v_1^{*}. $$
Since $v_1-v_1^{*}=u_{0}-h(v_{0}^{*})>0$, it follows that $\lim_{n\to\infty} v_{n}>0$. By Lemma~\ref{lem:limits}, we conclude $\lim_{n\to\infty}v_{n}=\infty$.

Similarly, if $u_{0}<h^*(v_{0})$, then for all $n\in\N$, we have $u_{n}\le  u_{n}^{*}$ and $v_{n}\le  v_{n}^{*}$, with
$$ v_{n+1}-v^{*}_{n+1}\ \le\ v_{n}-v_{n}^{*}\ \le\ v_1-v_1^{*}\ <\ 0. $$
Thus, we deduce $\lim_{n\to\infty} v_{n}<0$, and hence $h^*(v_0)=h(v_0)$, completing the proof.
\end{proof}

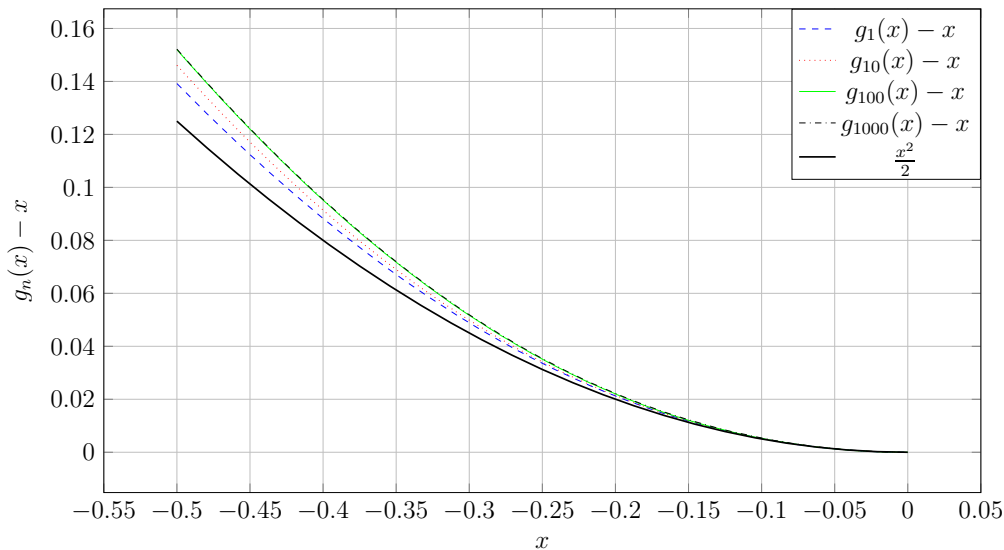
\begin{figure}[ht]
\begin{tikzpicture}[scale=0.8]
\begin{axis}[
    width=\textwidth,
    height=0.6\textwidth,
    xlabel={$x$},
    ylabel={$g_{n}(x)-x$},
    legend style={at={(1,1)}, anchor=north east},
    grid=both,
    scaled ticks=false,  
    ticklabel style={/pgf/number format/fixed},  
]
\addplot[color=blue, dashed] coordinates {
(-0.5000,0.1392) (-0.4949,0.1364) (-0.4899,0.1336) (-0.4848,0.1308) (-0.4798,0.1280) (-0.4747,0.1253) (-0.4697,0.1226) (-0.4646,0.1200) (-0.4596,0.1173) (-0.4545,0.1147) (-0.4495,0.1121) (-0.4444,0.1096) (-0.4394,0.1070) (-0.4343,0.1045) (-0.4293,0.1021) (-0.4242,0.0996) (-0.4192,0.0972) (-0.4141,0.0948) (-0.4091,0.0925) (-0.4040,0.0901) (-0.3990,0.0878) (-0.3939,0.0856) (-0.3889,0.0833) (-0.3838,0.0811) (-0.3788,0.0790) (-0.3737,0.0768) (-0.3687,0.0747) (-0.3636,0.0726) (-0.3586,0.0705) (-0.3535,0.0685) (-0.3485,0.0665) (-0.3434,0.0645) (-0.3384,0.0626) (-0.3333,0.0607) (-0.3283,0.0588) (-0.3232,0.0570) (-0.3182,0.0552) (-0.3131,0.0534) (-0.3081,0.0516) (-0.3030,0.0499) (-0.2980,0.0482) (-0.2929,0.0465) (-0.2879,0.0449) (-0.2828,0.0433) (-0.2778,0.0417) (-0.2727,0.0402) (-0.2677,0.0386) (-0.2626,0.0372) (-0.2576,0.0357) (-0.2525,0.0343) (-0.2475,0.0329) (-0.2424,0.0315) (-0.2374,0.0302) (-0.2323,0.0289) (-0.2273,0.0276) (-0.2222,0.0264) (-0.2172,0.0251) (-0.2121,0.0240) (-0.2071,0.0228) (-0.2020,0.0217) (-0.1970,0.0206) (-0.1919,0.0195) (-0.1869,0.0185) (-0.1818,0.0175) (-0.1768,0.0165) (-0.1717,0.0155) (-0.1667,0.0146) (-0.1616,0.0137) (-0.1566,0.0129) (-0.1515,0.0120) (-0.1465,0.0112) (-0.1414,0.0104) (-0.1364,0.0097) (-0.1313,0.0090) (-0.1263,0.0083) (-0.1212,0.0076) (-0.1162,0.0070) (-0.1111,0.0064) (-0.1061,0.0058) (-0.1010,0.0053) (-0.0960,0.0047) (-0.0909,0.0043) (-0.0859,0.0038) (-0.0808,0.0034) (-0.0758,0.0029) (-0.0707,0.0026) (-0.0657,0.0022) (-0.0606,0.0019) (-0.0556,0.0016) (-0.0505,0.0013) (-0.0455,0.0010) (-0.0404,0.0008) (-0.0354,0.0006) (-0.0303,0.0005) (-0.0253,0.0003) (-0.0202,0.0002) (-0.0152,0.0001) (-0.0101,0.0001) (-0.0051,0.0000) (0.0000,0.0000) };
\addlegendentry{$g_1(x)-x$}
\addplot[color=red, dotted] coordinates {
(-0.5000,0.1461) (-0.4949,0.1430) (-0.4899,0.1400) (-0.4848,0.1370) (-0.4798,0.1340) (-0.4747,0.1311) (-0.4697,0.1281) (-0.4646,0.1253) (-0.4596,0.1224) (-0.4545,0.1196) (-0.4495,0.1168) (-0.4444,0.1141) (-0.4394,0.1114) (-0.4343,0.1087) (-0.4293,0.1061) (-0.4242,0.1035) (-0.4192,0.1009) (-0.4141,0.0983) (-0.4091,0.0958) (-0.4040,0.0934) (-0.3990,0.0909) (-0.3939,0.0885) (-0.3889,0.0861) (-0.3838,0.0838) (-0.3788,0.0815) (-0.3737,0.0792) (-0.3687,0.0770) (-0.3636,0.0748) (-0.3586,0.0726) (-0.3535,0.0705) (-0.3485,0.0684) (-0.3434,0.0663) (-0.3384,0.0642) (-0.3333,0.0622) (-0.3283,0.0603) (-0.3232,0.0583) (-0.3182,0.0564) (-0.3131,0.0546) (-0.3081,0.0527) (-0.3030,0.0509) (-0.2980,0.0492) (-0.2929,0.0475) (-0.2879,0.0458) (-0.2828,0.0441) (-0.2778,0.0425) (-0.2727,0.0409) (-0.2677,0.0393) (-0.2626,0.0378) (-0.2576,0.0363) (-0.2525,0.0348) (-0.2475,0.0334) (-0.2424,0.0320) (-0.2374,0.0306) (-0.2323,0.0292) (-0.2273,0.0279) (-0.2222,0.0267) (-0.2172,0.0254) (-0.2121,0.0242) (-0.2071,0.0230) (-0.2020,0.0219) (-0.1970,0.0208) (-0.1919,0.0197) (-0.1869,0.0186) (-0.1818,0.0176) (-0.1768,0.0166) (-0.1717,0.0156) (-0.1667,0.0147) (-0.1616,0.0138) (-0.1566,0.0129) (-0.1515,0.0121) (-0.1465,0.0113) (-0.1414,0.0105) (-0.1364,0.0097) (-0.1313,0.0090) (-0.1263,0.0083) (-0.1212,0.0077) (-0.1162,0.0070) (-0.1111,0.0064) (-0.1061,0.0058) (-0.1010,0.0053) (-0.0960,0.0048) (-0.0909,0.0043) (-0.0859,0.0038) (-0.0808,0.0034) (-0.0758,0.0029) (-0.0707,0.0026) (-0.0657,0.0022) (-0.0606,0.0019) (-0.0556,0.0016) (-0.0505,0.0013) (-0.0455,0.0011) (-0.0404,0.0008) (-0.0354,0.0006) (-0.0303,0.0005) (-0.0253,0.0003) (-0.0202,0.0002) (-0.0152,0.0001) (-0.0101,0.0001) (-0.0051,0.0000) (0.0000,0.0000) };
\addlegendentry{$g_{10}(x)-x$}
\addplot[color=green, solid] coordinates {
(-0.5000,0.1520) (-0.4949,0.1488) (-0.4899,0.1457) (-0.4848,0.1426) (-0.4798,0.1395) (-0.4747,0.1365) (-0.4697,0.1334) (-0.4646,0.1305) (-0.4596,0.1275) (-0.4545,0.1246) (-0.4495,0.1217) (-0.4444,0.1189) (-0.4394,0.1160) (-0.4343,0.1133) (-0.4293,0.1105) (-0.4242,0.1078) (-0.4192,0.1051) (-0.4141,0.1024) (-0.4091,0.0998) (-0.4040,0.0972) (-0.3990,0.0947) (-0.3939,0.0922) (-0.3889,0.0897) (-0.3838,0.0872) (-0.3788,0.0848) (-0.3737,0.0824) (-0.3687,0.0801) (-0.3636,0.0778) (-0.3586,0.0755) (-0.3535,0.0733) (-0.3485,0.0711) (-0.3434,0.0689) (-0.3384,0.0668) (-0.3333,0.0647) (-0.3283,0.0626) (-0.3232,0.0606) (-0.3182,0.0586) (-0.3131,0.0566) (-0.3081,0.0547) (-0.3030,0.0528) (-0.2980,0.0510) (-0.2929,0.0492) (-0.2879,0.0474) (-0.2828,0.0456) (-0.2778,0.0439) (-0.2727,0.0423) (-0.2677,0.0406) (-0.2626,0.0390) (-0.2576,0.0374) (-0.2525,0.0359) (-0.2475,0.0344) (-0.2424,0.0329) (-0.2374,0.0315) (-0.2323,0.0301) (-0.2273,0.0288) (-0.2222,0.0274) (-0.2172,0.0261) (-0.2121,0.0249) (-0.2071,0.0237) (-0.2020,0.0225) (-0.1970,0.0213) (-0.1919,0.0202) (-0.1869,0.0191) (-0.1818,0.0180) (-0.1768,0.0170) (-0.1717,0.0160) (-0.1667,0.0150) (-0.1616,0.0141) (-0.1566,0.0132) (-0.1515,0.0123) (-0.1465,0.0115) (-0.1414,0.0107) (-0.1364,0.0099) (-0.1313,0.0092) (-0.1263,0.0085) (-0.1212,0.0078) (-0.1162,0.0071) (-0.1111,0.0065) (-0.1061,0.0059) (-0.1010,0.0054) (-0.0960,0.0048) (-0.0909,0.0043) (-0.0859,0.0039) (-0.0808,0.0034) (-0.0758,0.0030) (-0.0707,0.0026) (-0.0657,0.0022) (-0.0606,0.0019) (-0.0556,0.0016) (-0.0505,0.0013) (-0.0455,0.0011) (-0.0404,0.0008) (-0.0354,0.0006) (-0.0303,0.0005) (-0.0253,0.0003) (-0.0202,0.0002) (-0.0152,0.0001) (-0.0101,0.0001) (-0.0051,0.0000) (0.0000,0.0000) };
\addlegendentry{$g_{100}(x)-x$}
\addplot[color=black, dashdotted] coordinates {
(-0.5000,0.1522) (-0.4949,0.1490) (-0.4899,0.1459) (-0.4848,0.1428) (-0.4798,0.1397) (-0.4747,0.1367) (-0.4697,0.1337) (-0.4646,0.1307) (-0.4596,0.1277) (-0.4545,0.1248) (-0.4495,0.1219) (-0.4444,0.1191) (-0.4394,0.1163) (-0.4343,0.1135) (-0.4293,0.1107) (-0.4242,0.1080) (-0.4192,0.1053) (-0.4141,0.1027) (-0.4091,0.1000) (-0.4040,0.0975) (-0.3990,0.0949) (-0.3939,0.0924) (-0.3889,0.0899) (-0.3838,0.0875) (-0.3788,0.0850) (-0.3737,0.0827) (-0.3687,0.0803) (-0.3636,0.0780) (-0.3586,0.0757) (-0.3535,0.0735) (-0.3485,0.0713) (-0.3434,0.0691) (-0.3384,0.0670) (-0.3333,0.0649) (-0.3283,0.0628) (-0.3232,0.0608) (-0.3182,0.0588) (-0.3131,0.0569) (-0.3081,0.0549) (-0.3030,0.0530) (-0.2980,0.0512) (-0.2929,0.0494) (-0.2879,0.0476) (-0.2828,0.0459) (-0.2778,0.0441) (-0.2727,0.0425) (-0.2677,0.0408) (-0.2626,0.0392) (-0.2576,0.0377) (-0.2525,0.0361) (-0.2475,0.0346) (-0.2424,0.0332) (-0.2374,0.0317) (-0.2323,0.0303) (-0.2273,0.0290) (-0.2222,0.0276) (-0.2172,0.0263) (-0.2121,0.0251) (-0.2071,0.0239) (-0.2020,0.0227) (-0.1970,0.0215) (-0.1919,0.0204) (-0.1869,0.0193) (-0.1818,0.0182) (-0.1768,0.0172) (-0.1717,0.0162) (-0.1667,0.0152) (-0.1616,0.0143) (-0.1566,0.0134) (-0.1515,0.0125) (-0.1465,0.0117) (-0.1414,0.0109) (-0.1364,0.0101) (-0.1313,0.0093) (-0.1263,0.0086) (-0.1212,0.0079) (-0.1162,0.0073) (-0.1111,0.0067) (-0.1061,0.0061) (-0.1010,0.0055) (-0.0960,0.0050) (-0.0909,0.0044) (-0.0859,0.0040) (-0.0808,0.0035) (-0.0758,0.0031) (-0.0707,0.0027) (-0.0657,0.0023) (-0.0606,0.0020) (-0.0556,0.0017) (-0.0505,0.0014) (-0.0455,0.0011) (-0.0404,0.0009) (-0.0354,0.0007) (-0.0303,0.0005) (-0.0253,0.0004) (-0.0202,0.0002) (-0.0152,0.0001) (-0.0101,0.0001) (-0.0051,0.0000) (0.0000,0.0000) };
\addlegendentry{$g_{1000}(x)-x$}
\addplot[color=black,  thick, domain=-0.5:0] {x^2/2};
\addlegendentry{$\frac{x^2}{2}$}
\end{axis}
\end{tikzpicture}
\caption{\small Numerical computations  of $g_{n}(x)-x$ and $x^2/2$,  where $\eta=-0.5, K=10,$ and $\Psi(x)= \frac{1+2x}{1+x}$  corresponding to the generalized DR model described in Section~\ref{def-LF} with ${\tt Z}=1$ and $p=0.5$.}
\end{figure}

We conclude this section with the proof of Theorem~\ref{thm:criticalCurve}.

\begin{proof}[Proof of Theorem~\ref{thm:criticalCurve}]
For each $A>0$, we apply Proposition~\ref{p:g} to construct the unique function $g^A$ on $[-A,\infty)$ that satisfies \eqref{gxN}. By the compatibility property, for any $B>A$, the restriction of $g^B$ to $[-A,\infty)$ equals $g^A$. Thus, we can construct $g$ on $\R$ by taking the projective limit.

Next, define the function $h^{*} : x \mapsto g(x)-x$. We deduce from Proposition~\ref{p:g} that this is the unique nonincreasing $1$-Lipschitz, non-trivial solution to the functional equation  \eqref{eqn:functionalEquationCriticalCurve}, and by the same result, we know that $h^{*}(x) \sim x^2/2$ as $x \uparrow 0$. Finally, for any $(u_{0},v_{0})\in\R_+\times \R$, we have the following characterizations:
\begin{align*}
&(u_{0},v_{0})\in\mathcal{C} \iff u_{0}=h^{*}(v_{0}), \quad (u_{0},v_{0})\in\mathcal{P} \iff u_{0}>h^{*}(v_{0})\\
\text{and}\quad &(u_{0},v_{0})\in\mathcal{U} \iff u_{0}<h^{*}(v_{0}).
\end{align*}
This proves that $h=h^{*}$ and establishes the decomposition in \eqref{eqn:caracterisationOfPCU}.
\end{proof}

The duality relationship outlined in Proposition~\ref{prop:duality} enables the definition of a curve $\check{h}$, which plays the same role as $h$ in the backward evolution of the dynamics. Heuristically, it can be described as the trajectory of $(u_{n},v_{n})$ such that $(u_{0},v_{0})$ lies within a small neighborhood of $(0,0)$.

\begin{corollary}
\label{cor:dualCriticalCurve}
Assume \eqref{ass:psi}.  There exists a unique nondecreasing and continuous function $\check{h}$, with $0\le \check{h}(x)\le  \Psi(x)\, x_{+} $ for all $x\in\R$, that satisfies the functional equation
\begin{equation}\label{eqn:dualFunctionalEquation}
\check{h}(x+\check{h}(x))\ =\ \Psi(x+\check{h}(x)) \check{h}(x)\quad\text{for all }x\in  \R_+,
\end{equation}
with the asymptotics $\check{h}(x)\sim x^2/2$ as $x\downarrow 0$, and $\check{h}(x)\to \infty$ as $x\to \infty$.
\end{corollary}

\begin{proof}
Using the notation of Proposition~\ref{prop:duality}, we first apply Proposition~\ref{p:g} to the function
$$ \check{\Psi} : x \mapsto \frac{1}{\Psi(-x)}. $$
This defines a function $\tilde{h}$ such that if $(\check{u}_{k}, \check{v}_{k})_{k\ge 0}$ is the solution to \eqref{eqn:recursiveEquationReducedB} with $\check{\Psi}$ in place of $\Psi$, and with initial condition $\check{u}_{0}=\widetilde{h}(\check{v}_{0})$, then $\check{u}_{k}=\tilde{h}(\check{v}_{k})$ for all $k\ge 1$.  Let $N\ge 2$. Applying the duality relationship from Proposition~\ref{prop:duality} to $u_{n}:= \check{u}_{N-n}$ and $v_{n}:= -\check{v}_{N-n+1}$  for $0\le n \le N$, we obtain $u_{n}= \widetilde{h}(\check{v}_{N-n})=\widetilde{h}(-v_{n+1})$. Since $u_{n+1}=u_{n} \Psi(v_{n+1})$, we get
$$ u_{n+1}\,=\,\check{h}(v_{n+1}),\quad\text{where }\check{h}(x)\,:=\,\widetilde{h}(-x ) \Psi(x) \text{ for }x\in\R. $$
Thus, $\check{h}$ satisfies \eqref{eqn:dualFunctionalEquation} and exhibits regularity similar to that of $\widetilde{h}$ near $0$.

In addition, since $\check{h}$ is non-decreasing, it converges to a limit $\ell \in (0,\infty]$.  Assume that $\ell < \infty$. Then, letting $x \to \infty$ in \eqref{eqn:dualFunctionalEquation} and using $\Psi(\infty) > 1$, we obtain
\[
  \ell\,=\,\Psi(\infty) \ell\,>\,\ell,
\]
a contradiction. Hence $\lim_{x \to \infty} \check{h}(x) = \infty$.

 Finally, we prove uniqueness of the solution to \eqref{eqn:dualFunctionalEquation} as follows. Let $\widehat{h}$ be another solution to \eqref{eqn:dualFunctionalEquation}, satisfying the same regularity conditions as $\check{h}$. We aim to show that the function $x\to\widehat{h}(-x)/\Psi(-x)$ defines the critical curve associated with $\check{\Psi}$. By uniqueness of the critical curve, this implies $\widehat{h}(-x)/\Psi(-x)= \widetilde{h}(x)$, and therefore $\widehat{h} = \check h$. To this end, we verify that the function $f(y):= \widehat{h}(-y)/\Psi(-y)$ satisfies the equation
$$ f(y+f(y))\ =\ \check{\Psi}(y+f(y)) f(y), \quad y\le 0. $$
Note that the function $x\mapsto x+\widehat{h}(x)$ is increasing on $[0, \infty)$ and tends to $\infty$ as $x\to \infty$. For each $y\le 0$, let $x=x(y)$ be the unique nonnegative number such that  $y= -(x+\widehat{h}(x))$.
From \eqref{eqn:dualFunctionalEquation}, we have $\widehat{h}(-y)=\Psi(-y) \widehat{h}(x)$, which implies $f(y)=\widehat{h}(-y)/ \Psi(-y)=\widehat{h}(x)$.
Next, observe that $y+ f(y)= y+ \widehat{h}(x)= -x$, by the definition of $y$. Therefore
$$ f(y+f(y))\ =\ f(-x)\ =\  \frac{\widehat{h}(x)}{ \Psi(x)}. $$
Using the facts that $\widehat{h}(x)= f(y)$ and $1/\Psi(x)=\check{\Psi}(-x)=  \check{\Psi}(y+f(y))$, we conclude that
$$ f(y+f(y))\ =\ \check{\Psi}(y+f(y)) \, f(y), $$
showing that $f$ is the critical curve associated with $\check{\Psi}$. This completes the proof.
\end{proof}

\begin{remark} \label{rem:dual}\rm
The above proof shows that if $(u_{n}, v_{n})_{n\ge 0}$ is a solution of \eqref{eqn:recursiveEquationReducedB} satisfying $v_{0}\ge 0$ and $u_{n}= \check{h}(v_{n})$ for all $n\ge 0$, then for any $N \ge 2$, the dual system defined by
$$ (\check{u}_{n},\check{v}_{n})\ :=\ (u_{N-n},- v_{N-n+1})\quad\text{for }0\le n \le N, $$
is also a solution to \eqref{eqn:recursiveEquationReducedB}, but with $\check{\Psi}$ in place of $\Psi$. Moreover, we have $\check{u}_{n}= \widetilde h(\check{v}_{n})$ for $0\le n \le N$, where $\widetilde h(x):= \check{h}(-x) \check{\Psi}(x)$. In other words, the dual system $(\check{u}_{n},\check{v}_{n})_{0\le n \le N}$ moves along the critical curve associated with $\check{\Psi}$. This observation will be helpful in the proof of Lemma \ref{lem:dual}.
\end{remark}

The regularity of the critical curve $h$ will play a crucial role in the proof of the Derrida-Retaux conjecture in the nearly supercritical regime. In this section, we prove that the function $h$ is $\mathcal{C}^{1}$ and convex in a neighborhood of $0$, with a Lipschitz first derivative. Similarly to the previous section, we establish this result for the function $g$ satisfying \eqref{gxN}, by analyzing its approximation sequence.

\begin{lemma}\label{l:reg-g-min}
Assume \eqref{ass:psi}, and let $g$ be the unique nontrivial solution to \eqref{gxN}. For all $b>1$, there exists $\eta>0$ such that $g$ is convex and $\mathcal{C}^{1}$ on $[-\eta,0]$, with $g'$ being $b$-Lipschitz on $[-\eta,0]$.
\end{lemma}

\begin{proof}
Let $(g_{n})_{n\ge 1}$ be the sequence of functions defined inductively by \eqref{gn+1'} in the proof of Proposition~\ref{p:g}. Recall that $g_1$ is the solution of the equation $y'=\Psi(y)$ with initial condition $y(0)=0$. Therefore, $g_1$ is convex and $\mathcal{C}^2$, with $g_1(0)=0$ and $g_1'(0)=g_1''(0)=1$.

Let $b \in (1,4/3)$ and set $a=2-b \in (2/3,1)$. We will prove by induction that there exists $\eta>0$ such that
\begin{equation}\label{eqn:inducAim}
g_{n}''(x)\,\in\, [a,b] \quad \text{for all }x \in [-\eta,0].
\end{equation}
We first choose  $\delta>0$  such that $g_1'(x) \in [a,b]$ for all $x \in [-\delta,0]$. The value of $\eta\le\delta$ will be determined later. Additionally, define $C=\sup_{x \in [-\delta,0]} |\Psi''(x)|$. If necessary, we reduce $\delta$ so that $C\delta<1$.

Now assume that \eqref{eqn:inducAim} holds for some fixed $n\in\N$. By direct integration, we observe that $g_{n}'(x) \in [1+bx,1+ax]$ holds for all $x \in [-\eta,0]$. Moreover, since $|\Psi''|$ is bounded by $C$, we have
$$ \Psi'(x) \in [1+Cx,1-Cx]\quad\text{and}\quad\Psi(x) \in [1+ x-C x^2/2,1+x+C x^2/2]. $$
Next, differentiating \eqref{gn+1'} twice yields the equation
\begin{equation}\label{eqn:parts}
(K+1) g_{n+1}''(x)\ =\ g_{n}''(g_{n}(x))g_{n}'(x)^2+g_{n}'(g_{n}(x)) g_{n}''(x)+g_{n}''(x) U_{n}(x)+g_{n}'(x) V_{n}(x),
\end{equation}
where we have defined
\begin{align*}
  U_{n}(x) &:= K-\Psi(g_{n}(x))-( g_{n}(x)-x)\Psi'(g_{n}(x)),\\
  V_{n}(x) &:= 2 (1-g_{n}'(x)) \Psi'(g_{n}(x))-(g_{n}(x)-x) g_{n}'(x)\Psi''(g_{n}(x)).
\end{align*}

\textit{Bounding $U_{n}$}. We first bound $U_{n}$ using the known bounds for $\Psi$. We have:
\begin{align*}
U_{n}(x)\ &\le\ K-1-g_{n}(x)+\frac{C g_{n}(x)^2}{2}-(g_{n}(x)-x) (1+C g_{n}(x))\\
&\le\ K-1+x-2 g_{n}(x)+C x g_{n}(x),
\end{align*}
and similarly
$$ U_{n}(x)\ \ge\ K-1+x-2 g_{n}(x)-C x g_{n}(x). $$
Let us introduce $D>0$ such that for all $x \in [-\delta,0]$, we have
$$ C x g_{n}(x)\ \le\ C x g_1(x)\ \le\ D x^2. $$
Thus, we obtain
$$ K-1+x-2g_{n}(x)-D x^2\le  U_{n}(x)\le  K- 1+x-2 g_{n}(x)+ Dx^2. $$

\textit{Bounding $V_{n}$}. Next, we bound $V_{n}(x)$ in a similar manner, using the bounds on $\Psi'$, $\Psi''$. We write:
$$  2(1 -g_{n}'(x)) (1+C g_{n}(x))-\frac{C b x^2}{2}\ \le\ V_{n}(x)\ \le\  2 (1-g_{n}'(x))(1-C g_{n}(x))+\frac{C b x^2}{2}. $$
Since $1-g_{n}'(x) \in [-ax,-bx]$ and $g_{n}(x) \in [x+ax^2/2,x+bx^2/2]$, we conclude that there exists $E>0$, depending only on $C,a,b$, such that
$$  2 (1-g_{n}'(x)) -E x^2\ \le\ V_{n}(x)\ \le\ 2 (1-g_{n}'(x))+E x^2. $$

\textit{Bounding $g''_{n+1}$ from above}. Since $g_{n}''(x)\le  b$, equation\eqref{eqn:parts} yields
\begin{multline*}
(K+1)g_{n+1}''(x)\ \le\ b g_{n}'(x)^2+b g_{n}'(g_{n}(x))\\
+b (K-1+x-2 g_{n}(x))+2 g_{n}'(x) (1-g_{n}'(x))+R x^2\ =:\ P_{n}(x)
\end{multline*}
where $R>0$ is a sufficiently large constant, depending only on $a,b,C,D$ and $E$.
  We observe that $P_{n}(0)=b(K+1)$, and that $P_{n}$ is non-increasing on $[-\eta_2,0]$ for some $\eta_2> 0$ small enough. To see this, we compute its derivative:
\begin{equation*}
  P_{n}'(x) = 2 b g_{n}'(x)g_{n}''(x)+b g_{n}''(g_{n}(x)) g_{n}'(x)+b (1-2 g_{n}'(x))+2 g_{n}''(x)(1-2 g_{n}'(x))+2 R x.
\end{equation*}
Using the inductive hypothesis that \eqref{eqn:inducAim} holds for $g_n''$, we have $g_n'(x) \in [1 + bx,1+ax]$ for all $x \in [-\eta,0]$, thus
\[
  P_{n}'(x) \ge 2 b g_{n}''(x)+b g_{n}''(g_{n}(x))-b-2 g_{n}''(x)+S x,
\]
where $S$ is a constant depending on $R,a,b$. We rewrite this inequality as
\begin{equation*}
P_{n}'(x)\ge 2(b-1)g_{n}''(x)+ab-b+ S x\ge 2 (b-1)a+ab-b+S x\ \ge\ 3ab-2 a-b+S x.
\end{equation*}
Since $a=2-b$ and $b \in (1,4/3)$, we have $3a b-2 a-b=(b-1)(4-3b)>0$. Thus, we can choose $\eta_2>0$ small enough, depending only on $a,b, C,D$ such that $3 a b-2 a-b+S x\ge 0$ for all $x \in [-\eta_2,0]$. We conclude that $P_{n}$ is nondecreasing on $[-\min(\eta, \eta_2),0]$, and recall that $P_{n}(0)=(K+1)b$. Therefore, for all $x \in [-\min(\eta,\eta_2),0]$, we have
$$  (K+1)g_{n+1}''(x)\ \le\ P_{n}(x)\ \le\ P_{n}(0)\ =\ (K+1)b. $$

\textit{Bounding $g''_{n+1}$ from below}. The lower bound is treated in the same manner for $x \in [-\eta,0]$. Using $g''(x)\ge a$, we obtain from \eqref{eqn:parts} the following lower bound for $g_{n+1}''$
\begin{multline*}
(K+1) g_{n+1}''(x)\ \ge\ a g_{n}'(x)^2+a g_{n}'(g_{n}(x))\\
+a (K-1+x-2 g_{n}(x))+2 g_{n}'(x)(1-g_{n}'(x))-Rx^2\ =:\ Q_{n}(x),
\end{multline*}
where $R>0$ is a sufficiently large constant, depending only on $a,b,D$ and $E$. We note that $Q_{n}(0)=a(K+1)$. Next, we compute the derivative of $Q_{n}(x)$:
\begin{gather*}
Q_{n}'(x) = 2a g_{n}'(x) g_{n}''(x)+ag_{n}''(g_{n}(x))g_{n}'(x)+a(1-2 g_{n}'(x))+2 g_{n}''(x)(1-2 g_{n}'(x))-2 R x.
\intertext{Using again that $g_{n}'(x) \in [1 +bx,1+ax]$, we can bound this expression, namely}
Q_{n}'(x) \le 2a g_{n}''(x)+a g_{n}''(g_{n}(x))-a-2 g_{n}''(x)-S x,
\end{gather*}
with a constant $S$ depending on $R,a,b$. Since $g''_{n}(g_{n}(x))\le  b$ and $g''_{n}(x)\ge a$ we get
\begin{align*}
Q_{n}'(x) 
\ \le\ 2 (a-1) a+a b-a-Sx\le  a(2 a+b-3)-Sx
\end{align*}
Using $a=2 -b$, we find that $a(2 a+b-3)=(2-b)(1-b) <0$ for all $b  \in (1,2)$. Therefore, we can choose $\eta_{3}>0$ small enough, again only depending on $a,b,C,D$ such that $Q_{n}'(x)\le 0$ for all $x \in [-\min(\eta, \eta_{3}),0]$. We conclude that $Q_{n}$ is nonincreasing on $[-\min(\eta,\eta_{3}),0]$, and thus
$$ (K+1)g_{n+1}''(x)\ge Q_{n}(x)\ge Q_{n}(0)=(K+1)a. $$

Finally, fixing $\eta < \min(\eta_2,\eta_3)$, we observe that we have proved \eqref{eqn:inducAim} for $g''_{n+1}$.

\textit{Convexity of $g$}. Now, since \eqref{eqn:inducAim} holds for all $n\in\N$, the sequence $(g_{n}')_{n\ge 1}$ forms a family of continuous, increasing, and $b$-Lipschitz functions on $[-\eta,0]$ . By the Arzel\`a-Ascoli theorem, we can therefore extract a subsequence $(g_{n_{k}}')_{k\ge 1}$ that converges pointwise to a continuous, increasing, and $b$-Lipschitz limit function $f$.
Using the dominated convergence theorem, we then obtain
$$  g(x)=- \lim_{k\to\infty}  \int_{x}^{0} g'_{n_{k}}(y)\ \dd y\ =\ -\int_x^0 f(y)\ \dd y. $$
This shows that $g$ is $\mathcal{C}^{1}$ on $[-\eta,0]$ with $g'=f$. As $g'$ is increasing, we conclude that $g$ is convex. This completes the proof.
\end{proof}

\subsection{The critical regime}
\label{subsec:criticalEvolution}

In this subsection, we prove Theorem~\ref{thm:criticalEvolution}. For a fixed initial condition $(u_{0},v_{0})\in\mathcal{C}$, we examine the precise asymptotic behavior of $(u_{n},v_{n})$ as $n\to\infty$.

\begin{proof}[Proof of Theorem~\ref{thm:criticalEvolution}]
Let $v_{0}<0$, and fix $u_{0}=h(v_{0})$. Let $(u_{n},v_{n})_{n\ge 0}$ the solution to \eqref{eqn:recursiveEquationReducedB}. From Proposition~\ref{p:g}, we know that $u_{n}=h(v_{n})$ for all $n\ge 1$, and that
$$ \lim_{n\to\infty} u_{n}\ =\ \lim_{n\to\infty}v_{n}\ =\ 0. $$
Therefore, using the asymptotic relation $h(x)\sim x^2/2$ as $x\uparrow 0$, we immediately obtain $u_{n}\sim 2/n^2$ once we establish that $v_{n}\sim -2/n$.

To this end, let $0<w_1<\frac{1}{2}<w_2$ and $\delta>0$ such that $w_1 x^2\le  h(x)\le  w_2 x^2$ for all $x \in [-\delta,0]$. Thus, for sufficiently large $n$, we have
$$ v_{n+1}\,=\,v_{n}+u_{n}\,=\,v_{n}+h(v_{n})\,\in\, [v_{n}+w_1 v_{n}^2,v_{n}+w_2 v_{n}^2], $$
which implies the inequality
$$ \frac{1}{v_{n+1}}\ \le\ \frac{1}{v_{n}(1+w_1v_{n})}\ \le\ \frac{1}{v_{n}}\left(1-w_1 v_{n} \right)\ =\ \frac{1}{v_{n}}-w_1, $$
using $\frac{1}{1+x}\ge 1-x$ for all $x>-1$. Hence, we deduce that
$$ \limsup_{n\to\infty} \frac{1}{n v_{n}}\ \le\ -w_1. $$
Similarly, for $w_2'>w_2$, there exists an $n$ large enough such that
$$ \frac{1}{v_{n+1}}\ \ge\ \frac{1}{v_{n}(1+w_2 v_{n})}\ \ge\ \frac{1}{v_{n}}-w_2', $$
which gives $\liminf_{n\to\infty}1/nv_{n}\ge-w_2'$. We conclude that $\lim_{n\to\infty} nv_{n}=-2$, completing the proof.
\end{proof}

We conclude this section by applying the duality relationship to analyze the time it takes for the sequence $(u_{n},v_{n})$ to evolve from the first time when $v_{n}\ge 0$ to the first time when $u_{n}\ge 1$. For all initial points $(u_{0},v_{0})\in\mathcal{P}$, we define
\begin{equation}\label{eqn:defn*}
n_{*} =n_{*}(u_{0}, v_{0}):= \inf\{n\ge 0 : v_{n}\ge 0\text{ and }u_{n}\ge 1\}.
\end{equation}

\begin{lemma}
\label{lem:dual}
Let $v_{0}<0$ and $u_{0}> h(v_{0})$.  Defining $N_{0}':=\inf\{n\ge 0: v_{n}\ge  0\}$, there exists a positive constant $c>0$ such that
$$ (n_{*}-k)_+\,\le\,\frac{c}{v_{k}}\quad\text{for any }k>N_{0}'. $$
\end{lemma}

\begin{proof}
Recall the function $\check{h}$ from Corollary~\ref{cor:dualCriticalCurve}. We begin by verifying that $u_{N_{0}'}>\check{h}(v_{N_{0}'})$.  If $v_{N_{0}'}=0$, this is trivial since $\check{h}(0)=0$. If $v_{N_{0}'}>0$, there exists a unique $x>0$ such that $v_{N_{0}'}=x+ \check{h}(x)$. By definition, we have
$$ v_{N_{0}'}\,=\,v_{N_{0}'-1}+u_{N_{0}'-1}\,<\,u_{N_{0}'-1}, $$
which implies $u_{N_{0}'-1}> \check{h}(x)$ and
$$ u_{N_{0}'}\,=\,u_{N_{0}'-1} \Psi(v_{N_{0}'})\,>\,\check{h}(x)\Psi(x+ \check{h}(x))\,=\,\check{h}(x+ \check{h}(x))\,=\,\check{h}(v_{N_{0}'}). $$
By induction, we get that $u_{n}>\check{h}(v_{n})$ for any $n\ge N_{0}'$.

Now consider $k>N_{0}'$. If $u_{k} \ge 1$, $n_{*}\le k$ and there is nothing to prove. Otherwise, assume that $u_{k}<1$, thus $\check{h}(v_{k})<1$. Then $v_{k}$ is bounded from above, because $\check{h}(x)\to \infty$ as $x\to\infty$.

We now examine the system $(\widetilde{u}_{j}^{*}, \widetilde{v}_{j}^{*})_{j\ge 0}$, where $(\widetilde{u}^{*}_{0},\widetilde{v}^{*}_{0})=(\check{h}(v_{k}),v_{k})$ and
$$ (\widetilde{u}_{j}^{*},\widetilde{v}_{j}^{*})\,=\,(\check{h}(\widetilde{v}_{j}^{*}),\widetilde{v}^{*}_{j-1}+ \widetilde{u}^{*}_{j-1})\quad\text{for }j\ge 1. $$
This system satisfies the recursive equation \eqref{eqn:recursiveEquationReducedB} with the same function $\Psi$. 
An induction shows that $u_{j+k}\ge\widetilde{u}_{j}^{*}$ and $v_{j+k} \ge \widetilde{v}_{j}^{*}$ for all $j\ge 0$. Therefore, $n_{*}-k \le \underline{n}_{k}$, where
$$ \underline{n}_{k}  := \inf\{j \ge 0:  \widetilde{u}_{j}^{*}\ge 1\}. $$
To complete the proof, it therefore suffices that ${\underline{n}_{k}}\le c/v_{k}$ for some positive constant $c$. The idea is to apply the duality in Proposition~\ref{prop:duality} to
$$ (\check u_{j}, \check v_{j})_{0\le j\le \underline{n}_{k}}\ :=\ (\widetilde{u}^{*}_{\underline{n}_{k}-j}, -\widetilde{v}^{*}_{\underline{n}_{k}-j+1})_{0\le j \le \underline{n}_{k}}. $$
By Remark \ref{rem:dual}, the system $(\check u_{j}, \check v_{j})_{j\ge 0}$ is a solution to  \eqref{eqn:recursiveEquationReducedB} with $\check{\Psi}(x):= 1/\Psi(-x), x\in \R,$ replacing $\Psi$. Moreover,  $\underline{n}_{k}$ is the time it takes for this system to evolve along its critical curve from the initial position $(\widetilde{u}^{*}_{\underline{n}_{k}}, - \widetilde{v}^{*}_{\underline{n}_{k}+1})$  to $(\widetilde{u}^{*}_{0}, -\widetilde{v}^{*}_1)$.

Next, observe that $\widetilde{v}^{*}_1 > \widetilde{v}^{*}_{0}= v_{k}$. If we can show that $\widetilde{v}^{*}_{\underline{n}_{k}+1}\le c'$ for some positive~constant $c'$, then by the monotonicity in the starting point (see Lemma \ref{lem:monotonicity}), we have
$$ \underline{n}_{k}\ \le\ \inf\{j\ge 0: \check{v}_{j} \ge -v_{k}\}, $$
where $(\check u_{j}, \check v_{j})_{j\ge 0}$ now lies  on the critical curve with $\check{v}_{0}=-c'$. We obtain  ${\underline{n}_{k}}\le c/v_{k}$ for some positive constant $c$ by applying Theorem~\ref{thm:criticalEvolution} to the critical system $(\check u_{j}, \check v_{j})_{j\ge 0}$.

We are thus left with proving that $\widetilde{v}^{*}_{\underline{n}_{k}+1}\le c'$. Since $\widetilde{v}^{*}_{\underline{n}_{k}+1}= \widetilde{v}^{*}_{\underline{n}_{k}} +\check{h}(\widetilde{v}^{*}_{\underline{n}_{k}})$, it suffices to show that $\widetilde{v}^{*}_{\underline{n}_{k}}$ is bounded from above. By the definition of $\underline{n}_{k}$, we have $\check{h}(\widetilde{v}^{*}_{\underline{n}_{k}-1})= \widetilde{u}^{*}_{\underline{n}_{k}-1}$ $<1$. Since $\check{h}(x)\to \infty$ as $x\to \infty$, it follows that $\widetilde{v}^{*}_{\underline{n}_{k}-1}$ must be bounded from above. Furthermore, note that
$$ \widetilde{v}^{*}_{\underline{n}_{k}}\,=\,\widetilde{v}^{*}_{\underline{n}_{k}-1}+\widetilde{u}^{*}_{\underline{n}_{k}-1}\,<\,\widetilde{v}^{*}_{\underline{n}_{k}-1} +1. $$
Consequently, $\widetilde{v}^{*}_{\underline{n}_{k}}$ is bounded from above by a constant. This completes the proof.
\end{proof}

\section{The Derrida-Retaux conjecture}
\label{sec:DRconj}

The primary goal of this section is to prove Theorem~\ref{thm:main}, which establishes the Derrida-Retaux conjecture for the recursive equation \eqref{eqn:recursiveEquationReducedB}. This result, in turn, implies that the free energy of both solvable Derrida-Retaux models described in Section~\ref{sec:stochastique} undergoes an infinite-order BKT-type phase transition.

As a first step, we relate the free energy of $(u_{0},v_{0})$ to the number of steps $n_{*}$ required by the recursion \eqref{eqn:recursiveEquationReducedB} to bring $(u_{n},v_{n})$ into the domain $[1,\infty) \times \R_+$.

\begin{lemma}\label{lem:free}
Assume \eqref{ass:psi} and \eqref{ass:technical}. Then there exists a constant $c>0$ such that
$$ F(1,0)\Psi(\infty)^{-n_{*}}\ \le\ F(u_{0},v_{0})\ \le\ \max(u_{0},1)\Psi(\infty)^{-n_{*}+1}, $$
where $n_{*}$ is defined in \eqref{eqn:defn*}.
\end{lemma}

Observe that if $n_{*}$ is large, then $\log F(u_{0},v_{0})$ is comparable to $n_{*} \log \Psi(\infty)$, up to a correction of order $O(1)$. Consequently, the estimate for the free energy $\log F(h(v_{0})+\epsilon,v_{0})$ as $\epsilon\to 0$ reduces to a proper estimate for the asymptotic behavior of $n_{*}=n_{*}(h(v_{0})+\varepsilon, v_{0})$ as $\epsilon\to 0$.

\begin{proof}
Recall that $u_{n}/\Psi(\infty)^{n}$ is nonincreasing. Therefore, we have two possible cases:
\begin{enumerate}\itemsep2pt
\item $n_{*} =0$ and $F(u_{0},v_{0})\le u_{0}$.
\item $n_{*}>0$ and
$$ F(u_{0},v_{0})\ \le\ \frac{u_{n_{*}-1}}{\Psi(\infty)^{n_{*}-1}}\ \le\ \max(1,u_{0}) \Psi(\infty)^{-n_{*}+1}. $$
\end{enumerate}
Here, we used that either $u_{n_{*}-1}<1$ or $v_{n_{*}-1} <0$. In the second case, $u$ is decreasing for $k<n_{*}$ and thus $u_{n_{*}-1}<u_{0}$.

Next, we establish a lower bound for $F(u_{0},v_{0})$. We observe that
$$ F(u_{0},v_{0})\ =\ \lim_{n\to\infty} \frac{u_{n}}{\Psi(\infty)^{n}}\ =\ \lim_{n\to\infty} \frac{u_{n+k}}{\Psi(\infty)^{n+k}}\ =\ F(u_{k},v_{k}) \Psi(\infty)^{-k}. $$
Hence,
$$ F(u_{0},v_{0})\ \ge\ \Psi(\infty)^{-n_{*}} F(u_{n_{*}},v_{n_{*}})\ \ge\ \Psi(\infty)^{-n_{*}} F(1,0), $$
where the final inequality follows from the fact that $F$ is nondecreasing with respect to both $u$ and $v$, as stated in Lemma~\ref{lem:monotonicity}.
\end{proof}

Let $(u^{(\varepsilon)}_{n},v^{(\varepsilon)}_{n})$ be a solution of \eqref{eqn:recursiveEquationReducedB}, where $v^{(\varepsilon)}_{0}=v_{0} \le 0$ is fixed and $u^{(\varepsilon)}_{0}=h(v_{0})+\varepsilon$. We define
\begin{equation}
N_{0}^{(\varepsilon)}\ =\ N_{0}(u_{0}^{(\varepsilon)}, v_{0}^{(\varepsilon)})\ :=\ \max\{n\in\N : v^{(\varepsilon)}_{n} \le 0\}, \label{def-N0}
 \end{equation}
and observe that $u^{(\varepsilon)}_{N_{0}^{(\varepsilon)}} =\inf_{n\in\N} u^{(\varepsilon)}_{n}$, because $u^{(\varepsilon)}_{n+1}/u^{(\varepsilon)}_{n}=\Psi(v^{(\varepsilon)}_{n+1})\le  1$ if $n+1 \le N_{0}^{(\varepsilon)}$ and $u^{(\varepsilon)}_{n+1}/u^{(\varepsilon)}_{n}=\Psi(v^{(\varepsilon)}_{n+1})>1$ if $n+1>N_{0}^{(\varepsilon)}$.

The proof of Theorem~\ref{thm:main} will proceed in three main steps. First, for all $v_0\le 0$, we establish the existence of $c_{\star}(v_0)\ge 1$ such that
\begin{equation}\label{c-star}
\lim_{\varepsilon\to 0}\frac{1}{\varepsilon} u^{(\varepsilon)}_{N_{0}^{(\varepsilon)}}\ =\ c_{\star}(v_0).
\end{equation}

Note that   if $v_0=0$, then $N_0^{(\varepsilon)}=0$ by definition. Hence $u_0^{(\varepsilon)}=\varepsilon$, and \eqref{c-star} holds trivially with $c_\star(0)=1$.  For $v_0<0$, \eqref{c-star} constitutes a  critical step that allows us to relate the distance to the critical curve at the initial point of the evolution to its distance at the minimum, where the analysis is simpler (since the critical point is 0 at that stage).

For $A>0$, we observe that $n_{*}$, defined in \eqref{eqn:defn*}, can be decomposed as
$$ n_{*}\ =\ n^{(1)}_{A}+(n^{(2)}_{A}-n^{(1)}_{A})+(n_{*}-n^{(2)}_{A}), $$
where $n^{(1)}_{A}$ and $n^{(1)}_{A}$ are defined as
\begin{equation}\label{def-n1An2A}
  n^{(1)}_{A}\,:=\,\inf\{ n\in\N : v^{(\epsilon)}_{n}>-A \epsilon^{1/2} \}\quad\text{and}\quad
  n^{(2)}_{A}\,:=\,\inf\{ n\in\N : v^{(\epsilon)}_{n}> A \epsilon^{1/2} \}.
\end{equation}
The second step is to show that
\begin{align}\label{n2A-n1Aa}
     &\lim_{A\to\infty}\,\limsup_{\varepsilon\to 0}\, \big|(c_{\star}(v_0) \varepsilon)^{1/2}\left(n^{(2)}_{A}-N_0^{(\varepsilon)}\right) - \frac{\pi}{\sqrt{2}}\big|\ =\ 0,
\\
& \lim_{A\to\infty}\,\limsup_{\varepsilon\to 0}\, \big|(c_{\star}(v_0) \varepsilon)^{1/2}\left(N_0^{(\varepsilon)}-n^{(1)}_{A}\right) - \frac{\pi}{\sqrt{2}}\big|\ =\ 0 \qquad \text{if $v_0<0$.}\label{n2A-n1Ab}
\end{align}

\noindent This will rely on the fact that, on this time-interval, $(u_{n}^{(\epsilon)},v_{n}^{(\epsilon)})$ is well-approximated by an Eulerian scheme for the function $\tan$, with an initial condition given by \eqref{c-star}. Details will follow later, see \eqref{eq xvarepsilon}--\eqref{ODE solution}. Note that   if $v_0=0$, then $n^{(1)}_{A}=0$ by definition; hence the assumption $v_0<0$ in \eqref{n2A-n1Ab} is necessary.

Finally, we will prove that
\begin{equation}
\lim_{A\to\infty} \limsup_{\epsilon\to 0} \epsilon^{1/2}\left(n^{(1)}_{A}+(n_{*}-n^{(2)}_{A})\right)=0, \label{n1An2A}
\end{equation}
and thus show that $n_{*}$ is well-approached by $n^{(2)}_{A}-n^{(1)}_{A}$, for sufficiently  large $A$. The proof of Theorem~\ref{thm:main} will then be completed by applying Lemma~\ref{lem:free}.

The proofs of \eqref{c-star}, \eqref{n2A-n1Aa}--\eqref{n2A-n1Ab}, and \eqref{n1An2A} are provided in the following three steps. For simplicity, we will omit the superscript $(\varepsilon)$ in $u_{n}$, $v_{n}$, and $N_{0}$ where there is no risk of confusion.

\begin{proof}[Step 1: Proof of \eqref{c-star}] Only the case $v_0<0$ needs to be considered. We begin by proving a lemma that enables us to restrict our attention to the case where $v_{0}$ lies in any neighbourhood of $0$. For $\delta>0$, we define
\begin{equation}\label{def-n1delta}
n^{(3)}_{\delta}\,=\,n^{(3)}_{\delta}(u_{0}, v_{0})\,:=\,\inf\{n\in\N:v_{n}> -\delta\} .
\end{equation}

\begin{lemma}
\label{lem:easyStep} Assume \eqref{ass:psi}, and let $v_{0}<0$ and $u_{0}= h(v_{0})+\varepsilon$. Then
$$ \limsup_{\epsilon\to 0} n^{(3)}_{\delta}<\infty\quad\text{for any }\delta \in (0, |v_0|). $$
Moreover, for any sufficiently small $\delta>0$, there exists a constant $C=C(\delta)>0 $ such that
$$ u_{n^{(3)}_{\delta}}-h(v_{n^{(3)}_{\delta}})\ \sim\ C\varepsilon\quad\text{as }\varepsilon\to 0. $$
\end{lemma}

The second part of this lemma allows us to describe the relationship between the parameter $\epsilon=u_{0}-h(v_{0})$ and the corresponding parameter for the sequence $(u_{n^{(3)}_{\delta}+n}, v_{n^{(3)}_{\delta}+n})_{n\ge 0}$.

\begin{proof}
Recall from Lemma \ref{l:reg-g-min} that there exists some $\eta>0$ such that $h(x)=g(x)-x$ is $C^1$ on $[-\eta, 0]$. Let $0<\delta<\min(\eta, |v_0|)$. For each $k\in\N$, we consider $(u_{k}, v_{k})$ as a function of $(u_{0}, v_{0})$ and write $u_{k}=u_{k}(u_{0}, v_{0})$ and $v_{k}=v_{k}(u_{0}, v_{0})$. We observe that $(u_{k}(\cdot,\cdot),v_{k}(\cdot, \cdot))$, as an iteration of $\mathcal{C}^2$ functions, is also $\mathcal{C}^2$. Let $K=K(\delta)$ be the smallest positive integer such that
$$ v_K(h(v_{0}),v_{0})\,\ge \,-\delta. $$
Note that this constant $K$ does not depend on $\varepsilon > 0$. We will now prove that $\lim_{\epsilon \to 0} n_\delta^{(3)} = K$.
By the monotonicity of the $v_{k}$, we have $v_K(h(v_{0})+\varepsilon,v_{0})>-\delta$ for any $\varepsilon>0$. Moreover, by continuity,
$$ v_{K-1}(h(v_{0})+\varepsilon,v_{0})\,<\,-\delta\quad\text{for all sufficiently small }\varepsilon>0. $$
Thus, $n^{(3)}_{\delta}=K$ for all sufficiently small $\varepsilon$.  This completes the first part of   Lemma \ref{lem:easyStep}.

Next, for all sufficiently small $\varepsilon>0$,  we have
$$ u_{n^{(3)}_{\delta}}-h(v_{n^{(3)}_{\delta}})\ =\ u_K^{(\varepsilon)}- h(v_K^{(\varepsilon)}), $$
recalling that $(u^{(\varepsilon)}_{n},v^{(\varepsilon)}_{n})$ is the solution of \eqref{eqn:recursiveEquationReducedB} with $v^{(\varepsilon)}_{0}=v_{0}$ and $u^{(\varepsilon)}_{0}=h(v_{0})+\varepsilon$.

Let $k\ge 1$. Denote by $u_k^{\prime,(0)}$ and $v_k^{\prime,(0)}$ the derivatives of $u_k^{(\varepsilon)}$ and $v_k^{(\varepsilon)}$ with respect to $\varepsilon$ at $\varepsilon=0$.
It is easy to check that $u_k^{\prime,(0)} >0$ and $v_k^{\prime,(0)}>0$.  The function $\varepsilon \mapsto u_k^{(\varepsilon)}- h(v_k^{(\varepsilon)})$ is $C^1$ in a neighborhood of $0$, and vanishes there. Then for any $k\ge 1$,
\begin{equation} \label{def-Ck}   u_k^{(\varepsilon)}- h(v_k^{(\varepsilon)})\, \sim  \mathsf{C}_k  \, \varepsilon,\quad\text{as }\varepsilon\to 0,  \end{equation}
with $\mathsf{C}_k:= u_k^{\prime,(0)}- h'(v_k^{(0)}) v_k^{\prime,(0)}$.   Note that, if $k=0$, then \eqref{def-Ck}  holds trivially with $\mathsf{C}_0=1$. Since $h'(v_k^{(0)})\le 0$, it follows that $\mathsf{C}_k \ge u_k^{\prime,(0)} >0$. Applying \eqref{def-Ck} to $k=K$ completes the proof of Lemma \ref{lem:easyStep} with $C(\delta)=\mathsf{C}_{K(\delta)}$. \end{proof}

Let us introduce some preliminary notation and concepts. Recall the definition of $N_{0}$ from \eqref{def-N0}.  Specifically, we have
$$ N_{0}\ =\ N_{0}(u_{0},v_{0})\ =\ n^{(3)}_{\delta}+N_{0}(u_{n^{(3)}_{\delta}},v_{n^{(3)}_{\delta}}). $$
By considering the recursive system $(u_{ n^{(3)}_{\delta}+n}, v_{n^{(3)}_{\delta}+n})_{n\ge 0}$ and applying Lemma~\ref{lem:easyStep}, we can extend \eqref{c-star} to all $v_{0}<0$, with the constant $c_{\star}(v_0)$ potentially being multiplied by $C(\delta)$, provided we can establish \eqref{c-star} for $v_{0} \in (-\delta,0)$. In particular, using Lemma~\ref{l:reg-g-min}, we choose $\delta>0$ small enough such that $h$ is $\mathcal{C}^{1}$, convex with a Lipschitz continuous derivative $h'$ on $(-\delta,0)$. The precise value of $\delta$ will be determined later (see \eqref{eq:Dn} and \eqref{2ndconddelta}).

From this point on, we assume that $v_{0}\in (-\delta, 0)$, $0 < \varepsilon < \delta$   and $u_{0}-h(v_{0})=\epsilon$. Define
$$ \Delta_{n}\,:=\,u_{n}- h(v_{n})\quad \text{for }0\le n \le N_{0}. $$
Note that $\Delta_{n}>0$ for any $n\ge 0$ with the initial condition $\Delta_{0}=\epsilon$.

Let $(u^{*}_{n}, v^{*}_{n})$ denote the solution of \eqref{eqn:recursiveEquationReducedB} with $u^{*}_{0}:=h(v^{*}_{0})$ and $v^{*}_{0}=v_{0}$, which we refer to as the critical system.   When needed, we write $(u_n^*(v_0), v_n^*(v_0))$ to emphasize the dependence on $v_0$; otherwise, we write $(u_n^*, v_n^*)$. By comparison, we have $0>v_{n}\ge  v^{*}_{n}$ and we recall that $v^*_n \sim -\frac{2}{n}$ as $n \to \infty$. Moreover, as a function of $v_0$, $v^*_n$ is increasing.

\noindent This implies that there exists a positive constant $c$ such that for all $\delta\in (0,1), v_0\in (-\delta, 0)$,
\begin{equation}\label{vN0to0}
|v_{n}|\,\le \,|v^*_n|
\,\le\, \frac{c}{n+1}\quad\text{for all }0\le n \le N_{0}.
\end{equation}

We further note that by monotonicity, $N_{0}$ is an increasing function of $\epsilon$, and since  $\lim_{\epsilon\to 0}v_{n}$ $= v_{n}^{*}<0$ for each fixed $n$, we conclude $\lim_{\epsilon\to 0}N_{0}=\infty$.

In particular, the bound in \eqref{vN0to0} implies that $\lim_{\varepsilon\to0}v_{N_{0}}=0$, and using $v_{N_{0}+1}=u_{N_{0}}+v_{N_{0}}>0$, we see that $0\ge v_{N_{0}}>-u_{N_{0}}$. Since $h(x)= o(x)$ as $x\uparrow0$, we obtain
\begin{equation}\label{eq:uN0Delta}
\Delta_{N_{0}}\ =\ u_{N_{0}}-h(v_{N_{0}})\ \sim\ u_{N_{0}} \quad\text{as }\varepsilon\to 0.
\end{equation}
Therefore, to complete the proof of \eqref{c-star}, it suffices to prove the existence of a positive constant $c_{\star}(v_0)$ such that
\begin{equation}\label{DeltaN0-2}
\lim_{\varepsilon\to0}\frac{\Delta_{N_{0}}}{\varepsilon}\ =\ c_{\star}(v_0).
\end{equation}

The proof of \eqref{DeltaN0-2} involves on studying the variation of the $\Delta_{n}$. For  $n<N_{0}$, by definition, we have
$$ \Delta_{n+1}\ =\ u_{n} \Psi(v_{n+1})-h(v_{n+1})\ =\ (h(v_{n})+\Delta_{n}) \Psi(v_{n+1})-h(v_{n+1}). $$
Recall from Section~\ref{sec:criticalCurve} that we have set $g(x)=x+h(x)$ (for $x\le 0$). Therefore, we can express $v_{n+1}= u_{n}+ v_{n}= g(v_{n})+ \Delta_{n}$. Substituting into the expression for $\Psi(v_{n+1})$, a Taylor expansion provides
$$ \Psi(v_{n+1})\ =\ \Psi(g(v_{n}))+ \Psi'(g(v_{n}))\Delta_{n}+B_{n}\,\Delta_{n}^2, $$
with
$$ B_{n}\ :=\ \int_{0}^{1} (1-s) \Psi''(g(v_{n})+s\Delta_{n})\ \dd s $$
Now fix $b \in (1,4/3)$. By Lemma~\ref{l:reg-g-min}, we can choose $\delta$ small enough such that $h$ is $\mathcal{C}^{1}$, convex and $h'$ is $b$-Lipschitz. Then for any $ -\delta\le x \le y \le 0$ with $\delta \in (0,\eta)$, we have the estimate
\begin{equation}\label{h:quad}
0\ \le\ h(y)-[h(x)+h'(x) (y-x)]\ \le\ \frac{b}{2} (y-x)^2.
\end{equation}
Using this, we obtain the approximation
$$ h(v_{n+1})\ =\ h(g(v_{n}))+h'(g(v_{n}))\Delta_{n}+c_n\,\Delta_{n}^2 $$
for some $c_n\in [0,b/2]$. Recall that for $x\le 0$, we have $h(g(x))=\Psi(g(x))h(x)$. Thus, for any $n< N_{0}$, we arrive at the following expression for $\Delta_{n+1}$,
\begin{equation}\label{eq:Delatn+1}
\Delta_{n+1}\ =\ A_{n} \Delta_{n}+ D_{n}\,\Delta_{n}^2,
\end{equation}
where
\begin{align*}
A_{n}\ &:=\ h(v_{n}) \Psi'(g(v_{n}))+ \Psi(g(v_{n}))-h'(g(v_{n})),\\
D_{n}\ &:=\ h(v_{n}) B_{n}-c_n+\Psi'(g(v_{n}))+B_{n} \Delta_{n} \ =\ u_n B_n - c_n + \Psi'(g(v_{n})).
\end{align*}

We first treat $D_n$ and claim that one can choose (and fix) a sufficiently small $\delta>0$ such that for all $n < N_{0}$,
\begin{equation}\label{eq:Dn}
D_{n} \in [1/3, 2].
\end{equation}

Indeed, let $n< N_0$. For all $s \in [0, 1]$, we have $g(v_n) \le g(v_n) + s \Delta_n \le g(v_n) +\Delta_n= v_{n+1}\le 0$. By monotonicity of $g$, it follows that $g(v_n) \ge g(v_0)\ge g(-\delta)> -\delta$. Hence  $|B_{n}|\le\sup_{x\in[-\delta, 0]}|\Psi''(x)|$. Moreover, $u_n= u_{n-1} \Psi(v_n) \le u_{n-1}\le u_0 = \varepsilon + h(v_0) \le \varepsilon+ h(-\delta) \le \delta+ h(-\delta).$ By continuity of $h$, it follows that $u_n B_n  \to 0$ as $\delta\to 0$,  uniformly in $n < N_0$.    Since $\Psi'(0)=1$, we have  that $\Psi'(g(v_{n})) \to 1$ as $v_0\to 0$, uniformly in $n < N_0$.   Together with  $0\le c_n \le \frac{b}{2}< \frac23$,  this implies \eqref{eq:Dn}.

We further assume that $\delta$ is small enough to satisfy the following conditions:
\begin{equation}\label{2ndconddelta}
h(x)\,\ge\,\frac{x^2}{3}\quad\text{and}\quad b^2h(x)\,\le\,g'(x) x^2,\quad\text{for }x\in [-\delta, 0],
\end{equation}
which is possible since $h(x)\sim x^2/2$ and $g'(x)\to 1$ as $x\uparrow 0$ and $b^2<2$.

We now compare $A_{n}$ with $1$ by using the convexity of $h$. For $x< 0$, differentiating the expression $h(g(x))=\Psi(g(x))h(x)$ gives
$$ h'(g(x)) g'(x)\ =\ \Psi'(g(x))g'(x) h(x)+\Psi(g(x))h'(x). $$
Since $g'(x)=1+h'(x)$, we get
$$ A_{n}=\frac{\Psi(g(v_{n}))}{g'(v_{n})}= \frac{h(g(v_{n}))}{h(v_{n}) g'(v_{n})}. $$
Next, we examine the expression $\frac{h(g(x))}{h(x)}- g'(x)$. Using the definition of $g(x)=x+h(x)$, we can expand as follows:
$$ \frac{h(g(x))}{h(x)}- g'(x)\ =\ \frac{h(x+h(x))-h(x)}{h(x)}-h'(x)\ =\ h'(y_x)- h'(x) $$
for some $y_x\in [x, x+ h(x)]$. Since $h$ is convex, we have $h'(y_x)- h'(x)\ge 0$ and therefore conclude that
$$A_{n}\,\ge\,1. $$
This together with \eqref{eq:Dn} shows that $\Delta_{n}$ is increasing on $[0, N_{0}]$, in particular $\Delta_{n}\ge \Delta_{0}=\varepsilon$ for all $n \le N_{0}$.

We now turn to deriving an upper bound of $A_{n}$. Again, using \eqref{h:quad},
$$ h(g(x))\ =\ h(x+ h(x))\ \le\ h(x)+h'(x) h(x)+\frac{b}{2} h^2(x). $$
Thus, we obtain
$$ A_{n}-1\ =\ \frac{h(g(v_{n}))- h(v_{n}) g'(v_{n})}{h(v_{n}) g'(v_{n})}\ \le\ \frac{bh(v_{n})}{2 g'(v_{n})}\ <\ v_{n}^2, $$
where the last inequality follows from \eqref{2ndconddelta}. Using \eqref{eq:Delatn+1}, we get that for any $0\le k<N_{0}$,
\begin{equation}\label{eq:N0k}
\Delta_{N_{0}}\ =\ \Delta_{k}\,\prod_{n=k}^{N_{0}-1}(1+ (A_{n}-1)+D_{n} \Delta_{n}).
\end{equation}
This equation allows us to complete the proof of \eqref{c-star}.  Specifically, it follows from \eqref{def-Ck} that, for all (fixed) $k\in\Z_+$, there exists a constant $\mathsf{C}_k=\mathsf{C}_k(v_0)>0$ (recalling that $\mathsf{C}_0=1$) such that
$$
  \Delta_k\ \sim\ \mathsf{C}_k\,\varepsilon \quad \text{as }\varepsilon\to 0.
$$
Since $\Delta_n$ is increasing on $[0, N_0]$ and $N_0\to\infty$ as $\varepsilon\to \infty$,   it follows, for all $k \in \Z_+$, that $\mathsf{C}_{k} \le \mathsf{C}_{k+1}$ and
\begin{equation}\label{eqn:liminf}
\liminf_{\epsilon\to 0} \frac{\Delta_{N_{0}}}{\epsilon}\ \ge\ \mathsf{C}_{k}.
\end{equation}
On the other hand, using \eqref{eq:N0k} and noting that $A_{n}-1<v_{n}^2$ and $D_{n} \le 2$, we get the bound
$$ \Delta_{N_{0}}\le  \Delta_{k} \exp\left( \sum_{n=k}^{N_{0}-1} \left(v_{n}^2+2 \Delta_{n}\right) \right).
$$
We estimate the sum as follows: since $\sum_{n=k}^{N_{0}} \Delta_{n}\le \sum_{n=k}^{N_{0}} u_{n} \le v_{N_{0}}-v_{k}$, and $v_{n}^2\le (v_n^*)^2 \le \frac{c^2}{(n+1)^2}$ (from \eqref{vN0to0}),  there exists decreasing null sequence $(R_{n})_{n\ge 1}$ such that for all $k\in\Z_+$,
$$ \Delta_{N_{0}}\ \le\ \Delta_{k} e^{R_{k}+ 2\left(v_{N_{0}}-v_{k}\right)}, $$
with $R_k:= \sum_{n=k}^\infty (v_n^*)^2$.
Taking the limit $\epsilon\to 0$, and using $v_{N_{0}}\to 0$ and $v_{k}\to v_{k}^{*}$ as $\epsilon\to 0$, we obtain
\begin{equation}\label{eqn:limsup}
\limsup_{\epsilon\to 0} \frac{\Delta_{N_{0}}}{\epsilon}\ \le\ \mathsf{C}_{k} e^{R_{k}-2v^{*}_{k}}<\infty.
\end{equation}

By applying \eqref{eqn:liminf} to a generic $k$ and \eqref{eqn:limsup} with $k=0$, we arrive at the conclusion that $\mathsf{C}_{k}\le\mathsf{C}_0 e^{R_0-2v^*_0}=e^{R_0- 2 v_0}$ holds for all $k$. Consequently, the nondecreasing sequence $(\mathsf{C}_{k})_{k\ge 0}$ converges to a finite constant $c_{\star}(v_0)$. Moreover, since $\lim_{k\to \infty} (R_{k}-v^{*}_{k} )=0$, we can apply \eqref{eqn:liminf} and \eqref{eqn:limsup} to conclude that \eqref{DeltaN0-2} holds, thereby completing the proof.
\end{proof}

\begin{remark}\label{rem:c_star}\rm
Recall $\mathsf{C}_0=1$. We have shown that $c_\star(v_0)=\lim_{k\to\infty}\mathsf{C}_k\ \in\ [1,e^{R_0-2v_0}]$ for any $v_0<0$, where $R_0=\sum_{n=0}^{\infty}(v_n^*)^2$. We emphasize that $(v_n^*)_{n\ge 0}$ depends on $v_0$ through the initial condition $v_0^*=v_0$. For every fixed $n$, $v_n^*\to 0$ as $v_0\to 0$. By \eqref{vN0to0}, the dominated convergence theorem implies that $R_0\to 0$ as $v_0\to 0$; hence $\lim_{v_0\uparrow 0} c_\star(v_0)=1$.
\end{remark}

\begin{proof}[Step 2: Proofs of \eqref{n2A-n1Aa} and \eqref{n2A-n1Ab}]
The main idea here to approximate the function $\Psi$, in a neighborhood of $0$ of width $\epsilon^{1/2}$, by the simpler function $x \mapsto 1+x$. Specifically, we substitute this approximation into the recursion \eqref{eqn:recursiveEquationReducedB}, resulting in the simplified system:
\begin{equation}\label{eqn:recursionSimplified}
\begin{pmatrix} a_{n+1}\\ b_{n+1}\end{pmatrix}\ =\ \begin{pmatrix} a_{n} (1+b_{n+1})\\
b_{n}+a_{n}\end{pmatrix}.
\end{equation}
obtained by specifying \eqref{eqn:recursiveEquationReducedB} to the function $\Psi : x \mapsto 1+x$. This  simplified system has been considered in \cite{bz-notes}. 
Let for $\delta>0$,
\begin{equation}\label{def-n4}
n^{(4)}_{\delta}= n^{(4)}_{\delta}(u_{0}, v_{0}):= \inf\{n\in\N : v_{n}>\delta\}.
\end{equation}

\begin{lemma}\label{lem:simplification}
Assume \eqref{ass:psi}, let $\eta \in (0, 1)$ be small and fix $\delta>0$ such that
\begin{equation}\label{taylor-psi}
\frac{\Psi(x)-1}{x} \in [1-\eta,1+\eta]\quad\text{for all }x \in [0,\delta].
\end{equation}
Let $(u_{n},v_{n})_{n\ge 0}$ satisfy \eqref{eqn:recursiveEquationReducedB} with initial conditions $u_{0}>0$ and $v_{0} \in  (-u_{0}, 0]$. Then for all $ 1\le k< n^{(4)}_{\delta}$, we have
\begin{equation}
\frac{a_{k}^{(\eta, -)}}{1-\eta}\ \le\ u_{k}\ \le\ \frac{a_{k}^{(\eta,+)}}{1+\eta} \quad \text{and}\quad \frac{b_{k}^{(\eta, -)}}{1-\eta}\ \le\ v_{k}\ \le\ \frac{b_{k}^{(\eta, +)}}{1+\eta} ,\label{comparison-2sys}
\end{equation}
where $(a_{k}^{(\eta,\pm)}, b_{k}^{(\eta,\pm)})_{k\ge 0}$ are the solutions to the simplified recursion \eqref{eqn:recursionSimplified} with the respective initial conditions $(a_{0}^{(\eta, \pm)}, b_{0}^{(\eta, \pm)}):= ((1\pm \eta)  u_{0}, (1\pm \eta) v_{0})$.
\end{lemma}

Lemma~\ref{lem:simplification} enables us to compare the system $(u_{n},v_{n})_{n\ge 0}$ with $(a_{n}, b_{n})_{n\ge 0}$ up to a multiplicative factor $1\pm \eta$ on the initial conditions. In particular,  the evolution of the system $(u_{N_{0}+k},v_{N_{0}+k})_{k\ge 0}$ can be controlled by the two systems $(a_{k}^{(\eta,\pm)}, b_{k}^{(\eta,\pm)})_{k\ge 0}$ with initial conditions  $(a_{0}, b_{0})= ((1\pm\eta) u_{N_{0}}, (1\pm\eta) v_{N_{0}})$.

\begin{proof}
The proof follows from a direct application of Lemma~\ref{lem:monotonicity} with the functions
$$ \underline{\Psi}(x)=1+x-\eta |x| \quad\text{and}\quad\bar{\Psi}(x)=1+x+\eta |x|. $$
In particular, if $u_{0}>0$ and $v_{0} \in (-u_{0},0]$, we have (using the notation of that lemma)
\begin{align*}
\begin{pmatrix} \underline{u}_{n+1}\\ \underline{v}_{n+1}\end{pmatrix}\ =\ \begin{pmatrix} \underline{u}_{n} (1+(1-\eta) \underline{v}_{n+1})\\
\underline{v}_{n}+\underline{u}_{n}\end{pmatrix}
\quad \text{and}\quad
\begin{pmatrix} \bar{u}_{n+1}\\ \bar{v}_{n+1}\end{pmatrix}\ =\ \begin{pmatrix} \bar{u}_{n} (1+(1+\eta) \bar{v}_{n+1})\\
\bar{v}_{n}+\bar{u}_{n}\end{pmatrix}.
\end{align*}
We then observe that the sequence  $((1-\eta)^{-1} a_{n},(1-\eta)^{-1} b_{n})_{n\ge 0}$  satisfies the same recursion as $(\underline{u}_{n}, \underline{v}_{n})_{n\ge 0}$, and similarly $((1+\eta)^{-1} a_{n},(1+\eta)^{-1} b_{n})_{n\ge 0}$ follows the same recursion as $(\bar{u}_{n}, \bar{v}_{n})_{n\ge 0}$. This completes the argument for the direct application of Lemma~\ref{lem:monotonicity}.
\end{proof}

By \eqref{c-star}, we have $u_{N_{0}} \sim c_{\star}(v_0) \varepsilon$. Therefore the study of $n^{(2)}_{A}- N_{0}$ reduces to that of the corresponding quantity for the system $(a_{n}, b_{n})_{n\ge 0}$. For $N_{0}-n^{(1)}_{A}$ in the case $v_0<0$, we  consider the dual system $(u_{N_{0}-n}, v_{N_{0}+1-n})_{0\le n\le N_{0}}$ as in Proposition~\ref{prop:duality}, and note that the function $\check{\Psi}(x):=1/\Psi(-x)$  also satisfies \eqref{taylor-psi}. Thus, we can again apply Lemma~\ref{lem:simplification} to the dual system, reducing the study of $N_{0}- n^{(1)}_{A}$ to that of $(a_{n}, b_{n})_{n\ge 0}$. This explains why, in the case $v_0<0$, both $n^{(2)}_{A}- N_{0}$ and $N_{0}- n^{(1)}_{A}$ are of the same order in \eqref{n2A-n1Aa} and \eqref{n2A-n1Ab}.

From the previous discussion, the proofs of \eqref{n2A-n1Aa} and \eqref{n2A-n1Ab} reduce  to showing that for $(a_{n}, b_{n})_{n\ge 0}$ defined by \eqref{eqn:recursionSimplified} with initial condition $a_{0}=\varepsilon$ and $b_{0}\in (-\varepsilon, 0]$, we have, uniformly in $b_{0}\in (-\varepsilon, 0]$,
\begin{equation}\label{mA}
\lim_{A\to\infty} \,  \limsup_{\varepsilon\to 0} \big|\varepsilon^{1/2} m_{A}^{(\varepsilon)}\ - \ \frac{\pi}{\sqrt{2}}\big|=0,
\end{equation}
where
$$ m_{A}^{(\varepsilon)}\ =\ m_{A}^{(\varepsilon)}(a_{0}, b_{0})\ :=\ \inf\{n\ge 0: b_{n}>A \sqrt{\varepsilon}\}. $$
Note that, 
for all $a_0=\varepsilon, b_0\in (-\varepsilon, 0]$, $b_1>0$ and $a_1>\varepsilon$, hence
$$ m_{A}^{(\varepsilon)}(\varepsilon, 0)\ \le\ m_{A}^{(\varepsilon)}(a_{0}, b_{0})\ \le\ 1+ m_{A}^{(\varepsilon)}(\varepsilon,0) $$
by monotonicity. Therefore, it suffices to prove \eqref{mA} with $a_{0}=\varepsilon$ and $b_{0}=0$, which we assume from now on.

Next, define $x^{(\varepsilon)}_{n}=a_{n}/\epsilon$ and $y^{(\varepsilon)}_{n}=b_{n}/\epsilon^{1/2}$. We observe that
\begin{equation}\label{eq xvarepsilon}
\begin{pmatrix} x^{(\varepsilon)}_{n+1}-x^{(\varepsilon)}_{n}\\ y^{(\varepsilon)}_{n+1}-y^{(\varepsilon)}_{n}\end{pmatrix}\ =\ \begin{pmatrix} \epsilon^{1/2} x^{(\varepsilon)}_{n} y^{(\varepsilon)}_{n+1}\\ \epsilon^{1/2} x^{(\varepsilon)}_{n}\end{pmatrix}.
\end{equation}
with initial conditions $x^{(\varepsilon)}_{0}=1$ and $y^{(\varepsilon)}_{0} =0$. This is an Euler scheme for the differential system
\begin{equation}\label{Euler scheme}
\begin{pmatrix} x'\\ y'\end{pmatrix}\,=\,\begin{pmatrix} x y\\ x\end{pmatrix}.
\end{equation}
The solution to this system is given by
\begin{equation}\label{ODE solution}
x(t)\,=\,1+\tan(t/\sqrt{2})^2,\quad y(t)\,=\,\sqrt{2}\tan(t/\sqrt{2})\quad\text{for all }t<T\,:=\, \frac{\pi}{\sqrt{2}}.
\end{equation}
We then use that this Euler scheme converges uniformly to the solution of the ODE system. More precisely, this is a one-step method (see \cite[Chapter VIII]{Demailly}) with
\[
  \begin{pmatrix} x^{(\varepsilon)}_{n+1}\\ y^{(\varepsilon)}_{n+1}\end{pmatrix} = \begin{pmatrix} x^{(\varepsilon)}_{n}\\ y^{(\varepsilon)}_{n}\end{pmatrix} + \epsilon^{1/2} \Phi(x^{(\varepsilon)}_n,y^{(\varepsilon)}_n, \epsilon^{1/2}),
\]
where $\Phi(x,y,\epsilon^{1/2}) = \begin{pmatrix} x(y+\epsilon^{1/2}x)\\x\end{pmatrix}$ is continuous, locally Lipschitz in $(x,y)$, therefore forms a convergent one-step approximation scheme (see \cite{Demailly}, Remark on page 230).
\begin{fact}\label{fct:atExplosion}
For all $\delta>0$, we have
\begin{equation}\label{euler-1}
\limsup_{\varepsilon\to 0} \sup_{k\le  (T-\delta) \varepsilon^{-1/2}} \left( \left| x^{(\varepsilon)}_{k}-x_{k\varepsilon^{1/2}} \right|+\left| y^{(\varepsilon)}_{k}-y_{k\varepsilon^{1/2}} \right|\right)\ =\ 0.
\end{equation}
\end{fact}

Since $m_{A}(\varepsilon)= \inf\{n\ge 0: y^{(\varepsilon)}_{n}>A\}$, it follows from \eqref{euler-1} that
$$ \liminf_{A\to\infty} \liminf_{\varepsilon\to 0} \varepsilon^{1/2} m_{A}(\varepsilon) \ge  \frac{\pi}{\sqrt{2}} -\delta\quad\text{for any }\delta>0. $$
Letting $\delta\to 0$ gives the lower bound. For the upper bound, note that for all $t<\pi/\sqrt{2}$,
\begin{equation}\label{eq:euler-bn}
\lim_{\varepsilon\to 0}\frac{b_{\lfloor t \varepsilon^{-1/2}\rfloor}}{\varepsilon^{1/2}}\ =\   \lim_{\varepsilon\to 0} y^{(\varepsilon)}_{\lfloor t \varepsilon^{-1/2}\rfloor}\ =\ \frac{\tan(t\sqrt{2})}{\sqrt{2}}.
\end{equation}
For future reference, we also have the following limits for any $A>0$:
\begin{equation}\label{eq:euler-bn2}
\lim_{\varepsilon\to0} b_{m_{A}(\varepsilon)} / \varepsilon^{1/2}=A, \qquad \mbox{and} \quad \lim_{\varepsilon\to0} a_{m_{A}(\varepsilon)} / \varepsilon =1+A^2/2.
\end{equation}

By monotonicity, we deduce from \eqref{eq:euler-bn} that
$$  \lim_{\epsilon\to 0} b_{\lfloor  T \varepsilon^{-1/2}\rfloor}/\varepsilon^{1/2}=\infty. $$
This implies that
$$ \limsup_{A\to\infty} \limsup_{\varepsilon\to 0} \varepsilon^{1/2} m_{A}(\varepsilon) \le  \frac{\pi}{\sqrt{2}} .
$$
We have then proved \eqref{mA} and completed the proofs of \eqref{n2A-n1Aa} and \eqref{n2A-n1Ab}.
\end{proof}

\begin{proof}[Step 3: Proof of \eqref{n1An2A}] Let $v_0<0$.
We begin by considering $n^{(1)}_{A}$, comparing the supercritical system $(u_{n}, v_{n})_{n\ge 0}$ with the critical system. Let $(u^{*}_{n}, v^{*}_{n})_{n\ge 0}$ satisfy the recursive equation \eqref{eqn:recursiveEquationReducedB} with  $v^{*}_{0}:=v_{0}<0$ and $u^{*}_{0}:=h(v_{0})$. Since $u_{0}=h(v_{0})+\varepsilon>u^{*}_{0}$, Lemma~\ref{lem:monotonicity} implies that $v_{n} \ge  v^{*}_{n}$ for all $n$. Therefore, we have
$$ n^{(1)}_{A}\ \le\ \inf \{n\ge 1: v^{*}_{n}>- A \varepsilon^{1/2}\}. $$
By Theorem~\ref{thm:criticalEvolution}, we know that $v^{*}_{n} \sim -2/n$ as $n\to\infty$. Hence
$$ \limsup_{\varepsilon\to0}\varepsilon^{1/2} n^{(1)}_{A}\ \le\ \frac{2}{A}. $$

Next, we consider $n_{A}^{(2)}$, using \eqref{eq:euler-bn2} to obtain
$$ v_{n^{(2)}_{A}}\ \sim\ A \varepsilon^{1/2}\quad\text{as } \varepsilon\to 0. $$
An application of Lemma~\ref{lem:dual} provides the bound
$$ \limsup_{\varepsilon\to0}\varepsilon^{1/2} (n_{*}-n^{(2)}_{A})\ \le\  \frac{c}{A}. $$
Therefore, we have shown \eqref{n1An2A} in the case $v_0<0$. If $v_0=0$, then $n^{(1)}_A=0$ and \eqref{n1An2A} follow exactly in the same way using \eqref{eq:euler-bn2} and Lemma~\ref{lem:dual}.
\end{proof}

\begin{proof}[Proof of Theorem~\ref{thm:main}]
When $v_0<0$, we deduce from \eqref{n2A-n1Aa}, \eqref{n2A-n1Ab} and \eqref{n1An2A} that
$$ \lim_{\varepsilon\to0} \varepsilon^{1/2} n_{*}(h(v_{0})+\varepsilon, v_{0})=c_{\star}(v_0)^{-1/2} \pi \sqrt{2}. $$
The result then follows from Lemma~\ref{lem:free} with ${\mathbb C}_{v_0}:= c_{\star}(v_0)^{-1/2} \pi \sqrt{2}\,  \log \Psi(\infty)$. In the case $v_0=0$, using \eqref{n2A-n1Aa} and \eqref{n1An2A}, the corresponding result holds with ${\mathbb C}_0:= \frac{\pi}{\sqrt{2}} \log \Psi(\infty)$. By Remark \ref{rem:c_star}, $\lim_{v_0\to0} {\mathbb C}_{v_0}= 2 {\mathbb C}_0$ as claimed. This completes the proof of Theorem~\ref{thm:main}.
\end{proof}

{\noindent\bf Acknowledgments.} We are grateful to Zenghu Li and Zhan Shi for providing references and for stimulating discussions, especially for sharing the unpublished notes \cite{bz-notes}.   This project started 2019 during the conference  \textit{Branching in Innsbruck}; we warmly thank the organizers for their hospitality. We are also grateful to the anonymous referees for their careful reading and helpful comments, which improved the readability of the paper.

\medskip
{\noindent\bf Funding.}
The first author was partially supported by the German Research Foundation (DFG) under Germany's Excellence Strategy EXC 2044-390685587, Mathematics M\"unster: Dynamics--Geometry--Structure.

The second author was partially supported by ANR LOCAL (ANR-22-CE40-0012).

The third author was partially supported by the MITI interdisciplinary program 80PRIME GEx-MBB and the ANR MBAP-P (ANR-24-CE40-1833) project.

\end{document}